\documentclass[11pt,english]{amsart}
\usepackage[T1]{fontenc}
\usepackage[utf8]{inputenc}
\usepackage{geometry}
\usepackage{tikz}
\usetikzlibrary{patterns}
\usetikzlibrary{decorations.pathreplacing}
\usepackage{subcaption}
\usepackage{graphicx}
\usepackage{babel}
\addto\shorthandsspanish{\spanishdeactivate{~<>}}

\usepackage{enumitem}
\usepackage{color}
\usepackage{mathrsfs}
\usepackage{mathtools}
\usepackage{amsmath}
\usepackage{amsthm}
\usepackage{amssymb}
\usepackage{csquotes}
\usepackage[unicode=true,pdfusetitle,
 bookmarks=true,bookmarksnumbered=false,bookmarksopen=false,
 breaklinks=false,pdfborder={0 0 1},backref=false,colorlinks=false]
 {hyperref}
\usepackage[backend=biber,style=numeric,natbib=true,maxbibnames=99]{biblatex} 

\usepackage{scalerel,stackengine}
\stackMath
\newcommand\reallywidehat[1]{%
\savestack{\tmpbox}{\stretchto{%
  \scaleto{%
    \scalerel*[\widthof{\ensuremath{#1}}]{\kern-.6pt\bigwedge\kern-.6pt}%
    {\rule[-\textheight/2]{1ex}{\textheight}}
  }{\textheight}%
}{0.5ex}}%
\stackon[1pt]{#1}{\tmpbox}%
}
\makeatletter

\theoremstyle{plain} 
\newtheorem{teorema}{Theorem}[section]
\newtheorem{lema}[teorema]{Lemma} 
\newtheorem{coro}[teorema]{Corollary}
\newtheorem{prop}[teorema]{Proposition}

\theoremstyle{definition} 
\newtheorem{defin}[teorema]{Definition} 

\theoremstyle{remark}

\newtheorem{rem}[teorema]{Remark}

\newcommand{\val}{\mathbf{Val}}
\newcommand{\Pa}{\text{Pa}}

\usepackage{anysize}    

\marginsize{4cm}{4cm}{2cm}{2cm} 

\begin{document}

\title[Non-continuous valuations on convex bodies]{Non-continuous valuations on convex bodies and a new characterization of volume}

\author[J.~S. Ib\'a\~nez-Marcos]{Jorge S. Ib\'a\~nez-Marcos}
\address{Departamento de An\'alisis Matem\'atico y Matem\'atica Aplicada\\
Facultad de Matem\'aticas \\ Universidad Complutense de Madrid \\
Madrid 28040}
\email{jorgesib@ucm.es}

\author[P. Tradacete]{Pedro Tradacete}
\address{Instituto de Ciencias Matem\'aticas (CSIC-UAM-UC3M-UCM)\\
Consejo Superior de Investigaciones Cient\'ificas\\
C/ Nicol\'as Cabrera, 13--15, Campus de Cantoblanco UAM\\
28049 Madrid, Spain.}
\email{pedro.tradacete@icmat.es}

\author[I. Villanueva]{Ignacio Villanueva}
\address{Departamento de An\'alisis Matem\'atico y Matem\'atica Aplicada\\
Facultad de Matem\'aticas \\ Universidad Complutense de Madrid \\
Madrid 28040}
\email{ignaciov@mat.ucm.es}



\thanks{Research partially supported by grants PID2020-116398GB-I00, PID2024-162214NB-I00 and CEX2023-001347-S funded by MCIN/AEI/10.13039/501100011033. The first author has also received financial support from Ministerio de Ciencia, Innovación y Universidades through an FPU Grant FPU22/02055}

\begin{abstract}
This paper investigates the use of automatic continuity techniques in the context of valuations on convex bodies. We first provide an automatic continuity theorem for valuations restricted to parallelotopes with respect to a fixed basis. This result in combination with a counting argument provides a strengthened version of a classical characterization of volume due to Hadwiger. As a byproduct of the proof it is shown that $[0,n-1]\cup\{n\}$ are precisely the possible degrees of homogeneity of bounded translation invariant valuations on $n$-dimensional convex bodies. 
\end{abstract}

\maketitle
\section{Introduction}

The theory of valuations on convex bodies has become one of the central pillars of modern convex and integral geometry. Its origins can be traced back to the early 20th century, when Dehn proved a negative solution for Hilbert's Third Problem \cite{Dehn}. Early contributions by Hadwiger, most notably his characterization theorem in the 1950s, established a cornerstone of the field: every continuous, rigid motion invariant valuation on convex bodies in Euclidean space is a linear combination of intrinsic volumes (see \cite{Hadwiger}). This result revealed the deep structure underlying seemingly disparate geometric measures and opened the door to systematic classification results.

In the decades that followed, the theory of valuations underwent a remarkable expansion, influenced by methods from topology, measure theory, and representation theory. The introduction of continuous and smooth valuations by Alesker in the early 2000s \cite{Alesker1} marked a new phase in the subject, connecting convex geometry with areas such as integral geometry on manifolds and the theory of distributions.
As a result, valuation theory has become an active interface between convex geometry, geometric measure theory, and algebraic topology, with applications ranging from stochastic geometry to complex and Hermitian integral geometry (see, e.g. \cite{Bernig, Handbook, equivariantes, LibroHug, MonikayFabian, logconcavas, Schneider}). In the recent years, relevant classification results on valuation theory on different context have appeared (see \cite{Alesker1999,BernigFu2011,ColesantiLudwigMussnig2024,Haberl,Knoerr24,LudwigReitzner1999,LudwigReitzner2010,Estrellado2,Reticulos}).


The definition of valuation is ``algebraic'' in the sense that it does not impose any continuity conditions. Nevertheless, non-continuous valuations seem too unwieldy, and it becomes necessary to require some form of continuity in order to obtain positive characterization results. In the closely related framework of measure theory, the analogy would be considering $\sigma$-additive, which encodes a certain continuity property, or bounded measures, instead of general finitely additive measures whose behavior is beyond any reasonable description. It is therefore natural to study whether in the valuation setting continuity can be replaced by weaker requirements which, when coupled with the valuation condition, will suffice to guarantee characterization results. 

Previous related research includes several works where non continuous valuations on polytopes are considered (see, for example, \cite{Equidecomposability,Monika-noncontinuous,SLinvariantPolitopes}), and other works where the focus is on non continuous valuations defined on the convex bodies which have the origin as an interior point. In this case, additional symmetry assumptions such as SL$(n)$-invariance \cite{Haberl-Parapatits,LudwigReitzner2010} or SL$(n)$-covariance \cite{SLcovariance} are imposed. To the best of our knowledge, non continuous translation invaluations on the full space of convex bodies, $\mathcal{K}^n$, have not been previously studied.


In the direction of weakening continuity conditions, it is worth recalling the theory of automatic continuity, which emerged from the study of algebraic structures endowed with compatible topologies, where one seeks to understand when a purely algebraic property together with appropriate definability conditions actually suffice to guarantee continuity of certain maps. In its most classical form, the theory addresses the question of when a homomorphism between topological groups (see \cite{RosendalAC}), rings, or algebras must necessarily be continuous, even without explicit continuity assumptions. These foundational insights reveal that, in many analytical contexts, continuity is not an additional hypothesis but rather an inevitable consequence of the underlying structure. 


The theory of valuations and the theory of automatic continuity share a foundational element in their early development: The Cauchy Functional Equation. In the former, it underlies the basis for the definition of what is considered the first valuation, Dehn's valuation \cite{Dehn}. In the latter, it is the primordial example of the theory and the first example studied in depth \cite{Cauchy}, \cite{Hamel}. Although both theories share this important link, to the best of our knowledge not much attention has been paid to the use of automatic continuity in the theory of valuations. 

In this article, we investigate the applicability of techniques from automatic continuity theory to valuation theory, with the aim of replacing the continuity condition with weaker conditions, such as boundedness or Borel measurability, which might still guarantee some control on the behavior of valuations and allow us to find useful characterizations.

Our main result is a new version of the classical result by Hadwiger (see Theorem \ref{t:volume}) characterizing the volume in $\mathbb R^n$ as the unique translation invariant continuous and $n$-homogeneous valuation on $\mathcal{K}^n$, the $n$-dimensional convex bodies. In our result the continuity assumption is replaced by a boundedness condition. It is also worth noting that in this setting, without continuity assumptions, it is not known whether the inclusion-exclusion principle for valuations is valid, so our proof must circumvent this. Before stating the result, let $B$ denote the (euclidean) unit ball of $\mathbb{R}^n$ and $\mathcal{K}(B)=\{K\in\mathcal K^n:K\subset B\}$.

\begin{teorema}\label{t:main}
Let $V:\mathcal{K}^n\to\mathbb{R}$ be a translation invariant weakly additive map that is $n$-homogeneous and is bounded on $\mathcal{K}(B)$. Then, $V$ is proportional to the Lebesgue measure on $\mathbb R^n$.
\end{teorema}

Weakly additive maps (see Definition \ref{d:weak additivity}) behave like valuations on a restricted class of sets. In particular, valuations on convex sets are weakly additive. The proof of Theorem \ref{t:main} is based on a counting argument together with an automatic continuity result for valuations on a restricted class of convex bodies, the parallelotopes. Actually, the proof only requires $V$ to be $n$-homogeneous for rational coefficients.


Motivated by the previous result, we further investigate the structure of bounded translational invariant valuations on convex sets. A well-known decomposition result due to McMullen states that if a continuous translation invariant valuation on $\mathcal{K}^n$ is $m$-homogeneous, then $m$ must belong to the set $\{0, 1, \ldots, n\}$. In opposition to this, we will show that the degrees of homogeneity of bounded translation invariant valuations on $\mathcal{K}^n$ are precisely $[0,n-1] \cup \{n\}$.

Finally, some of the limitations of our approach are exhibited. Several counterexamples illustrate where automatic continuity techniques break down when extended beyond $\Pa(\{e_i\}_{i=1}^n)$, offering some insight into the boundaries of this method's effectiveness in the broader context of valuations on  convex sets.

The paper is structured as follows. After stating the definitions and known results from valuation theory and automatic continuity theory, in Section \ref{S:par} we focus on a restricted class of convex bodies in $\mathbb R^n$: paralelotopes with respect to a fixed basis, denoted by $\Pa(\{e_i\}_{i=1}^n)$. Within this setting, we establish an automatic continuity result for translation invariant valuations (Corollary \ref{c:automatic continuity}). This result will be crucial for the proof of the main result. Section \ref{S:main} is devoted to the proof of our main result, Theorem \ref{t:main}. The main technical tools are Theorem \ref{t:McMullen-pa} from Section \ref{S:par}, and a volume counting argument which we have isolated as Lemma \ref{l:counting}. In Section \ref{S:homogeneidad}, we prove the above mentioned result about the possible homogeneity degrees for bounded translation invariant valuations on convex sets. We use a family of bounded translation invariant $m$-homogeneous valuations for $m\in [0, n-1]$, and show how Lemma \ref{l:counting} implies that such valuations cannot exist for $m\in (n-1, n)$. In Section \ref{S:limitations} we show some of the limitations to the applicability of automatic continuity theory to valuation theory. Finally, in the Appendix we have included a general characterization of valuations on parallelotopes which are not necessarily translation invariant. In particular, it is shown that in that case, there can be no satisfactory automatic continuity results.

\section{Preliminaries}\label{S:preliminaries}

\subsection{Preliminaries on valuations}

Let $\mathcal{K}^n$ denote the set of convex bodies in $\mathbb{R}^n$, that is, the set of convex, compact and nonempty subsets of $\mathbb{R}^n$. The set $\mathcal{K}^n$ is equipped with the operations of Minkowski addition and scalar multiplication: Given $K,L\in\mathcal{K}^n$ and $\lambda\geq 0$, these operations are defined by
\[
    K+L=\{a+b:a\in K,\;b\in L\}, \quad \lambda K=\{\lambda a:a\in K\}.
\]
Additionally, we also equip $\mathcal{K}^n$ with the Hausdorff metric, denoted by $d_H$, which is given by
\[
    d_H(K,L)=\min\{\varepsilon\geq0: K\subset L+\varepsilon B, \;L\subset K+\varepsilon B\},
\]
where $B$ denotes the closed Euclidean unit ball. We refer to \cite{Schneider} as a standard reference on these notions.
Let us denote by $\mathcal{K}(B)$ the set of convex bodies lying inside the unit closed unit ball $B$, that is
\[
    \mathcal{K}(B)= \{K\in\mathcal{K}^n:K\subset B\}.
\]

In this work, we will consider two distinguished subclasses of convex bodies: the set of polytopes, denoted by $\mathcal{P}^n$, and the set parallelotopes associated with a basis, denoted by $\Pa(\{e_i\}_{i=1}^n)$, which we define next.

\begin{defin}
Given $\{e_1, \dots,e_n\}$ a basis of $\mathbb{R}^n$, let us define 
\[
    \Pa(\{e_i\}_{i=1}^n)=\left\{\sum_{i=1}^n [a_i,b_i]e_i : a_i\le b_i\right\}.
\]
\end{defin}

$\Pa(\{e_i\}_{i=1}^n)$ consists of all compact, non-empty parallelotopes whose edges are parallel to the coordinate axes determined by the basis $\{e_i\}_{i=1}^n$. This class includes non-proper parallelotopes, such as the singletons. Throughout this work, unless otherwise specified, the basis $\{e_i\}_{i=1}^n$ will be fixed. 


\begin{defin}
Let $\mathcal{S}$ be $\mathcal{K}^n,\mathcal{P}^n$, or $\Pa(\{e_i\}_{i=1}^n)$, and let $E$ be a vector space. We say that $V:\mathcal{S}\rightarrow E$ is a valuation if
\[
    V(K\cup L)+V(K\cap L)=V(K)+V(L),
\]
whenever $K\cup L\in \mathcal{S}$.
\end{defin}

Most of the valuations considered in this paper will be scalar-valued. However, we will also make use of area measures, denoted by $S_i$, which are valuations taking values on $\mathcal{M}(\mathbb{S}^{n-1})$, the space of Borel regular measures on the sphere. For a detailed account of these valuations, we refer to \cite[Chapter 4]{Schneider}. 

Although valuation is the standard notion in the literature, our results also hold if the mapping $V$ is weakly additive.

\begin{defin}[Weak additivity]\label{d:weak additivity}
Let $\mathcal{S}=\mathcal{K}^n,\mathcal{P}^n$, or $\Pa(\{e_i\}_{i=1}^n)$, and let $E$ be a vector space. We say that a map $V:\mathcal{S}\longrightarrow E$ is weakly additive if for any hyperplane $H$ with corresponding closed halfspaces $H^-$ and $H^+$ the relation
\[
    V(K)+V(K\cap H)=V(K\cap H^+)+V(K\cap H^-)
\]
holds whenever $K,K\cap H,K\cap H^+,K\cap H^-\in \mathcal{S}$.
\end{defin}

When working in $\mathcal{P}^n$ or in $\Pa(\{e_i\}_{i=1}^n)$, it is known that weak additivity is equivalent to the valuation property (see \cite[Theorem 6.2.3]{Schneider}). This is no longer the case for $\mathcal{K}^n$ (see \cite[Page 173]{Debiladitiva}); however, if $V$ is weakly additive on $\mathcal{K}^n$, then its restrictions, $V|_{\mathcal{P}^n}$ and $V|_{\Pa(\{e_i\}_{i=1}^n)}$, are valuations on their respective domains. This fact allows us to weaken the valuation hypothesis to the condition that $V$ is weakly additive.


Since general valuations may exhibit pathological behavior, additional geometrical or topological properties are typically imposed. Continuity (as maps from $\mathcal S$ with the restricted Hausdorf metric to $\mathbb R$) is a standard assumption. However, our main interest will be to weaken this continuity hypothesis. In addition, translation invariant valuations will also be essential here: a valuation $V:\mathcal{S}\to E$ is translation invariant if for any $K\in\mathcal{S}$ and any $x\in\mathbb{R}^n$ it satisfies
\[
    V(K+x)=V(K).
\]
Unless stated otherwise, all the valuations appearing in this work with will be translation invariant.

Recall also that a valuation $V$ is said to be $\alpha$-homogeneous if, for every $\lambda\geq 0$ and every $K\in\mathcal{S}$,
\[
    V(\lambda K)=\lambda^\alpha V(K).
\]
For $0\leq j\leq n$, let $\val_j$ denote the set of $j$-homogeneous, continuous, and translation invariant scalar-valued valuations on $\mathcal K^n$. Due to their fundamental strucutre, these valuations play a central role in valuation theory. In particular, they satisfy the Inclusion-Exclusion property:

\begin{defin}[Inclusion-Exclusion]
A valuation $V:\mathcal{S}\to E$ satisfies the Inclusion-Exclusion property if, for every $m\in\mathbb{N}$,
\[
    V\left(\bigcup_{i=1}^mK_i\right)=\sum_{r=1}^m(-1)^{r-1}\sum_{1\leq i_1\leq\dots\leq i_r\leq m}V\left(\bigcap_{j=1}^rK_{i_j}\right)
\]
whenever $K_1,\dots,K_m\in\mathcal{S}$ with $\bigcup_{i=1}^m K_i\in\mathcal{S}$.
\end{defin}

It is well known that all valuations defined on  $\mathcal{P}^n$ satisfy the Inclusion-Exclusion property (see \cite[Theorem 6.2.3]{Schneider}). In contrast, in the case of $\mathcal{K}^n$, it is known that all continuous valuations satisfy this (in fact, $\sigma$-continuity is enough, see \cite[Theorem 1.21]{sigmacont}). However, it remains an open question whether continuity can be completely dropped. Since our focus will be on valuations which are not necessarily continuous, we will not assume Inclusion-Exclusion.

Let us recall the fundamental volume characterization theorem for valuations due to Hadwiger \cite{Hadwiger}.

\begin{teorema}\label{t:volume}
Let $V\in \val_n$, then there exists $c\in\mathbb{R}$ such that 
\[
    V=c\text{Vol}, 
\] where \text{Vol} is the Lebesgue measure in $\mathbb{R}^n$.
\end{teorema}

Our main result (Theorem \ref{t:main}) is an extension of this theorem where we replace the continuity hypothesis by a (weaker) boundedness condition. We refer the reader to \cite{Hadwiger} for the original proof of Theorem \ref{t:volume} and to \cite[Theorem 3.1.1]{Alesker} for a classical proof using valuation theory. The original proof in \cite{Hadwiger} is based in the theory of contents while the proof in \cite{Alesker} is based in McMullen's decomposition theorem, which we recall next:

\begin{teorema}[McMullen's decomposition]
Let $V:\mathcal{P}^n\to\mathbb{R}$ be a translation invariant valuation. For $i\in\{0,\dots,n\}$, there exist $\phi_i:\mathcal{P}^n\to\mathbb{R}$ translation invariant valuation such that $\phi_i(\lambda K)=\lambda^i\phi_i(K)$ for any positive rational $\lambda $ and
\[
    V(P)=\sum_{i=0}^n\phi_i(P)
\]
for $P\in\mathcal{P}^n$.
\end{teorema}

Our proof of the main theorem also circumvents the use of McMullen's decomposition. In Section \ref{S:par}, we establish a simple analogue of this decomposition for valuations on $\Pa(\{e_i\}_{i=1}^n)$ (see Theorem \ref{t:McMullen-pa}). This proof relies only in elementary computations. In fact, the main point will be a simple geometrical observation given by Lemma \ref{lema:unioncuad}. This version of the Theorem, together with a few notions of automatic continuity and a counting argument via the appropriate grid of parallelotopes will be the key for the proof of Lemma \ref{l:counting}. 

Motivated by this, we explore whether the continuity assumptions can be relaxed in other classical results from valuation theory. Recall McMullen's characterization of $\val_{n-1}$ (see \cite{McMullen}):

\begin{teorema}[McMullen]\label{teo:McMullenn-1}
If $V\in\val_{n-1}$, then there exists $f\in C(\mathbb{S}^{n-1})$ such that
\[
    V(K)=\int_{\mathbb{S}^{n-1}}f\mathrm{d}S_{n-1}(K),
\]
where $S_{n-1}(K)$ denotes the surface area measure of $K$. Moreover, this function is unique up to addition with linear functions.
\end{teorema}
At the beginning of Section \ref{S:limitations}, we will see that the continuity hypothesis cannot be dropped in this case.

\subsection{Preliminaries on automatic continuity}
The starting point of Automatic Continuity Theory can be traced back to the Cauchy Functional Equation, more precisely to the study of which functions $f:\mathbb{R}\to\mathbb{R}$ satisfy 
\[
    f(x+y)=f(x)+f(y)
\]
for any $x,y\in\mathbb{R}$. Solutions to this equation are called additive functions. Clearly, all linear functions satisfy the Cauchy Functional Equation, but could these be the only solutions? Cauchy \cite{Cauchy} proved that if the function $f$ is continuous, then the answer is positive. Later, in 1906, Hamel \cite{Hamel} proved the existence of discontinuous additive functions by constructing a basis of $\mathbb{R}$ over $\mathbb{Q}$, now known as a Hamel Basis. These discontinuous additive functions are also known as Hamel functions. An important aspect of the Cauchy Functional Equation is the dichotomy in the nature of its solutions, either they are continuous and linear, or extremely pathological functions. This fact can be summarized in the following well-known theorem (see e.g. \cite{Kuczma}).
\begin{teorema} \label{teo:CFE}
Let $f:\mathbb R\rightarrow \mathbb{R}$ be an additive function. The following are equivalent:
\begin{itemize}
    \item[(1)] $f$ is linear
    \item[(2)] $f$ is continuous
    \item[(3)] $f$ is bounded (above or below) in any bounded interval
    \item[(4)] $f$ is Borel.
\end{itemize}
\end{teorema}

A very complete survey of automatic continuity on the more general setting of group homomorphisms can be found in \cite{RosendalAC} (see also \cite{RosendalPI} for more recent developments). Theorems in the spirit of \ref{teo:CFE} is what we refer to as an automatic continuity result, and although most developments usually focus on the Borel condition, our interest in the setting of valuations will shift to the boundedness condition.

We will also need to work with separately additive functions for which a version of Theorem \ref{teo:CFE} holds.
\begin{defin}
Let $G:\mathbb{R}^n\to\mathbb{R}$. We will say that $G$ is separately additive (or multiadditive) if for any $i\in\{1,\dots,n\}$ and any $a_1,\dots,a_{i-1},a_{i+1},\ldots,a_n\in\mathbb{R}$, the following is an additive function
\[
    x\longmapsto G(a_1,\dots,a_{i-1},x,a_{i+1},\dots,a_n).
\]
\end{defin}

\begin{lema}\label{lema:multiseparada}
Let $G:\mathbb{R}^n\to\mathbb{R}$ be a separately additive function. If $G$ is bounded on a set $U\subset\mathbb{R}^n$ that has nonempty interior (or it is Borel in $\mathbb{R}^n$) then $G$ is a separately linear function.
\end{lema}
\begin{proof}
We will prove it by induction on $n$. The case $n=1$ follows from Theorem \ref{teo:CFE}. Assume $n>1$, and let $I\subset \mathbb R$, $V\subset \mathbb R^{n-1}$ be open balls such that $I\times V\subset U$. Let $(x_{2},\dots,x_n)\in V$ and consider the mapping $g:\mathbb{R}\to\mathbb{R}$ defined as $g(y)=G(y,x_2,\dots,x_n)$. Then, $g$ is an additive function and it is bounded on the interval $I$, so it is linear by Theorem \ref{teo:CFE}. Thus, for any $(x_2,\dots,x_n)\in V$, we get 
\[
    G(y,x_2,\dots,x_n)=g(y)=\alpha_{x_2,\dots,x_n}y
\]
for any $y\in\mathbb{R}$. Since $G(I\times V)$ is bounded, it follows that $\alpha_{x_2,\dots,x_n}$ is also bounded for any $(x_2,\dots,x_n)\in V$. This implies that for any open and bounded interval $J\subset \mathbb{R}$ one has that $G(J\times C)$ is bounded.

Now, for any $x\in \mathbb{R}$, the function $G(x,\cdot,\dots,\cdot):\mathbb{R}^{n-1}\to\mathbb{R}$ is separately additive and it is bounded on the open set $V\subset\mathbb{R}^{n-1}$. By induction hypothesis, $G(x,\cdot,\dots,\cdot)$ is separately linear for any $x\in\mathbb{R}$. Therefore, for any $x\in \mathbb{R}$ we get
\[
    G(x,y_2,\dots,y_n)=\beta(x)\prod_{i=2}^ny_i,
\]
and since $G(J\times C)$ is bounded for any open and bounded interval $J\subset \mathbb{R}$, we get that $\beta(J)$ is bounded for any such $J$. Thus, $G(J\times D)$ is bounded for any such $J$ and any open and bounded set $D\subset\mathbb{R}^{n-1}$. Fix $(y_2,\dots,y_n)\in\mathbb{R}^{n-1}$ and consider again $g(\cdot)=G(\cdot,y_2,\dots,y_n)$. Now, for any open and bounded $D\subset\mathbb{R}^{n-1}$ such that $(y_2,\dots,y_n)\in D$, we get that $g(J)$ is bounded for any open and bounded interval $J$ and, by Theorem \ref{teo:CFE}, we get that $g$ is lineal. 

In the case that $G$ is Borel, the proof is straightforward.
\end{proof}

\section{Translation Invariant valuations on parallelotopes}\label{S:par}
In this section, we first provide a characterization of translation invariant valuations on $\Pa(\{e_i\}_{i=1}^n)$ without assuming any continuity. Although the proof can be considered elementary, we have not found any explicit reference. This characterization will prove useful and it will show that translation invariant valuations do satisfy some analog to Theorem \ref{teo:CFE} (see Theorem \ref{teo:valuaciones_continuidad}). We refer to \cite[Ch. 4]{Klain} for a study of continuous and translation invariant valuations on $\Pa(\{e_i\}_{i=1}^n)$. 

As noted earlier, weak additivity is equivalent to the valuation property on $\Pa(\{e_i\}_{i=1}^n)$. Consequently, all the results in this section may be reformulated using weak additivity instead of the valuation property, and in the case of Lemma \ref{lema:Simple}, the restatement is strictly stronger.



Recall that in the definition of valuation, one only needs to check the identity $V(K_1)+V(K_2)=V(K_1\cup K_2)+V(K_1\cap K_2)$ when the union $K_1\cup K_2$ is in the family of sets where the valuation is defined. In the case of $\Pa(\{e_i\}_{i=1}^n)$, if one wants to check whether a mapping is a valuation then one only needs to  consider the case when the union $K_1\cup K_2$ is in $\Pa(\{e_i\}_{i=1}^n)$. The following lemma states that in order for $K_1\cup K_2$ to be in $\Pa(\{e_i\}_{i=1}^n)$ either one of the sets is contained in the other (in which case the valuation equality holds trivially) or $K_1$ and $K_2$ differ in only one dimension. Therefore, to check that a mapping is a valuation on $\Pa(\{e_i\}_{i=1}^n)$, it suffices to verify the valuation property only in this second case. 

\begin{lema}\label{lema:unioncuad}
Let $K_1,K_2\in\Pa(\{e_i\}_{i=1}^n)$ such that $K_j=\sum_{i=1}^n [a_i^j,b_i^j]e_i$ with $a_i^j\leq b_i^j$ for any $i\in\{1,\dots,n\}$ and $j=1,2$. If $K_1\cup K_2\in\Pa(\{e_i\}_{i=1}^n)$, then exactly one of the following is satisfied:
\begin{itemize}
    \item[1)] $K_1\subset K_2$ or $K_2\subset K_1$.
    \item[2)] There exists exactly one $k\in\{1,\dots,n\}$ such that $a_i^1=a_i^2$ and $b_i^1=b_i^2$ for all $i\neq k$, $[a_{k}^1,b_{k}^1]\not\subset[a_{k}^2,b_{k}^2]$ and $[a_{k}^2,b_{k}^2]\not\subset [a_{k}^1,b_{k}^1]$.
\end{itemize}
\end{lema}

\begin{proof}
Let us assume $K_1\cup K_2\in\Pa(\{e_i\}_{i=1}^n)$. Hence, there exist real numbers $c_i\leq d_i$ for $i\in\{1,\dots,n\}$ such that
\[
     K_1\cup K_2=\sum_{i=1}^n[c_i,d_i]e_i.
\]
Note that this clearly implies that $c_i=\min\{a_i^1,a_i^2\}$, $d_i=\max\{b_i^1,b_i^2\}$ and that $[a_i^1,b_i^1]\cap[a_i^2,b_i^2]\neq\emptyset$ for all $i$.

Now assume that $K_1\not\subset K_2$ and $K_2\not\subset K_1$. Therefore, there are two indices $j$ and $k$ such that
\[
    [a_{j}^1,b_{j}^1]\not\subset [a_{j}^2,b_{j}^2] \quad \text{ and }\quad [a_{k}^2,b_{k}^2]\not\subset [a_{k}^1,b_{k}^1].
\]
We will see that if $j\neq k$ then we run into a contradiction. For simplicity, let us assume without loss of generality that $j=1$ and $k=2$. Let $z_i\in [a_i^1,b_i^1]\cap[a_i^2,b_i^2]$ for all $i\neq 1,2$ and take
\[
    x\in [a_{1}^1,b_{1}^1]\setminus [a_{1}^2,b_{1}^2]\quad \text{ and }\quad y\in [a_{2}^2,b_{2}^2]\setminus [a_{2}^1,b_{2}^1].
\]
Clearly, we have that
\[
    x\in [c_1,d_1],\quad y\in[c_2,d_2] \quad \text{and}\quad z_i\in[c_i,d_i]
\]
for any $i$. Thus, taking the point $p=(x,y,z_3,\dots,z_n)$, we get $p\in K_1\cup K_2$ but $p\not\in K_2$ since $x\not\in [a_1^2,b_1^2]$. Analogously, $p\not\in K_1$ making a contradiction and proving the result.
\end{proof}

This lemma will allow us characterize translation invariant valuations defined on $\Pa(\{e_i\}_{i=1}^n)$ as {\em separately affine additive} functions, defined below. 

\begin{defin}
Let $\mathbb{R}_+=\{x\in\mathbb{R}: x\geq 0\}$. We call $f:\mathbb{R}_+\to\mathbb{R}$ an affine additive function if for all $x,y\in\mathbb{R}_+$ we have
\[
    f(x+y)=f(x)+f(y)-f(0).
\]
Equivalently, an affine additive function is the sum of an additive function with a constant.

A function $F:\mathbb{R}^n_+\to\mathbb{R}$ is separately affine additive if for any $i\in\{1,\dots, n\}$ and any $a_1,\dots,a_n\in\mathbb{R}_+$, the function
\[
    x\mapsto F(a_1,\dots,a_{i-1},x,a_{i+1},\dots,a_n)
\]
is affine additive.
\end{defin}

These functions are essential because they codify all the valuation information. 

\begin{lema}\label{lema:ident.tras}
Let $V:\Pa(\{e_i\}_{i=1}^n)\to\mathbb{R}$ be a translation invariant mapping and let $F:\mathbb{R}^n_+\to \mathbb{R}$ be defined as
\[
    F(a_1,\dots,a_n)=V\left(\sum_{i=1}^n [0,a_i]e_i\right)
\]
for $a_1,\dots,a_n\geq 0$. Then $V$ is a valuation if and only if $F$ is separately affine additive.
\end{lema}

\begin{proof}
Let $V$ be a translation invariant valuation and let $F$ be defined as
\[
    F(a_1,\dots,a_n)= V\left(\sum_{i=1}^n[0,a_i]e_i\right)
\]
for any $a_1,\dots,a_n\geq 0$. Let $j\in\{1,\ldots,n\}$, fix $a_1,\dots,a_{j-1},a_{j+1},\ldots,a_n\geq 0$ and consider the function $f_j(x)= F(a_1,\dots,a_{j-1},x,a_{j+1},\ldots,a_n)$ for $x\geq 0$. To verify that $f_j$ is affine additive, let $x,y\geq 0$ and denote $P_j=\sum_{i\neq j}[0,a_i]e_i$. By the valuation property of $V$,
\begin{align*}
    f_j(x+y)&=V\left([0,x+y]e_j+P_j\right)\\
    &=V\left([0,x]e_j+P_j\right)+V\left([x,x+y]e_j+P_j\right)-V\left(\{xe_j\}+P_j\right).
\end{align*}
Now, since $V$ is translation invariant,
\begin{align*}
    f_j(x+y)&=V\left([0,x]e_j+P_j\right)+V\left([0,y]e_j+P_j\right)-V\left(\{0\}+P_j\right)\\
    &=f_j(x)+f_j(y)-f_j(0).
\end{align*}
Thus, $F$ is separately affine additive.

Conversely, let $F$ be defined as above and suppose that it is separately affine additive. In order for $V$ to be a valuation, it needs to satisfy
\[
    V(K_1)+V(K_2)=V(K_1\cup K_2)+V(K_1\cap K_2)
\]
whenever $K_1\cup K_2,K_1,K_2\in\Pa(\{e_i\}_{i=1}^n)$. Now, if $K_1\subset K_2$ or $K_2\subset K_1$ the equality is trivially obtained. By Lemma \ref{lema:unioncuad}, without loss of generality, we may assume that $K_1$ and $K_2$ coincide in all dimensions except the first one. Let $a_i\leq b_i$ for $i\neq 1$, $a_1^1\leq b_1^1$ and $a_1^2\leq b_1^2$ be real numbers such that
\[
    K_j=[a_1^j,b_1^j]e_1+ \sum_{i=2}^n[a_i,b_i]e_i \quad\text{for }j=1,2,
\]
and such that $K_1\cap K_2\neq\emptyset$. Since $K_1\not\subset K_2$ and $K_2\not\subset K_1$, me may assume without loss of generality that $a_1^1<a_1^2\leq b_1^1<b_1^2$. Therefore,
\[
    K_1\cup K_2=[a_1^1,b_1^2]e_1+\sum_{i=2}^n[a_i,b_i]e_i\quad\text{and}\quad K_1\cap K_2=[a_1^2,b_1^1]e_1+\sum_{i=2}^n[a_i,b_i]e_i.
\]
Thus, by the translation invariant property, 
\begin{align*}
    V(K_1\cup K_2)&=V\left([0,b_1^2-a_1^1]e_1+\sum_{i=2}^n[0,b_i-a_i]e_i\right)\\
    &=F(b_1^2-a_1^1,b_2-a_2,\dots,b_n-a_n),\\
    V(K_1\cap K_2)&=V\left([0,b_1^1-a_1^2]e_1+\sum_{i=2}^n[0,b_i-a_i]e_i\right)\\
    &=F(b_1^1-a_1^2,b_2-a_2,\dots,b_n-a_n).
\end{align*}
For fixed $a_i$ and $b_i$ for $i\geq 2$, let $f$ be defined as
\[
    f(x)= F(x,b_2-a_2,\dots,b_n-a_n),
\]
which yields 
\[
    V(K_1\cup K_2)+V(K_1\cap K_2)=f(b_1^2-a_1^1)+f(b_1^1-a_1^2).
\]
Since $f$ is affine additive, we have
\begin{align*}
    f(b_1^2-a_1^1)+f(a_1^1)-f(0)=f(b_1^2),\\
    f(b_1^1-a_1^2)+f(a_1^2)-f(0)=f(b_1^1).
\end{align*}
Therefore,
\begin{align*}
    V(K_1\cup K_2)+V(K_1\cap K_2)&=f(b_1^2)-f(a_1^1)+f(0) +f(b_1^1)-f(a_1^2)+f(0)\\
    &=f(b_1^2)-f(a_1^2)+f(0) +f(b_1^1)-f(a_1^1)+f(0)\\
    &=V(K_2)+V(K_1).
\end{align*}
\end{proof}

We show next that separately affine additive functions can be written as a sum of separately additive functions of different dimensions, see Proposition \ref{teo:descompadit}. To prove that, first we will show the following lemma.

\begin{lema}\label{lema:descomp1}
Let $F:\mathbb{R}_+^n\to\mathbb{R}$ be a separately affine additive function. Then, there exists a separately additive function $G:\mathbb{R}^{n}_+\to\mathbb{R}$ and, for every set $I\subset\{1,\dots,n\}$ with $I\neq \{1,\dots,n\}$, there is a separately affine additive function $F_I:\mathbb{R}_+^{|I|}\to\mathbb{R}$ (where $F_\emptyset$ is simply a constant) such that

\[
    F(a_1,\dots,a_n)=G((a_i)_{i\in\{1,\ldots,n\}})+\sum_{\substack{
         I\subset \{1,\dots,n\}  \\
         I\neq \{1,\dots,n\} 
    }} F_I((a_i)_{i\in I}).
\]
\end{lema}
\begin{proof}
For any subset $I\subset \{1,\dots,n\}$, we denote 
\[
    a^I=(a_i^I)_{i=1}^n \quad \text{where }a^I_i=\left\{\begin{array}{ll}
        a_i & \text{if }i\in I \\
        0 & \text{if }i\not\in I.
    \end{array}  \right.
\]
Define $G$ as follows:
\[
    G(a_1,\dots,a_n)= \sum_{I\subset\{1,\dots,n\}}(-1)^{|I|}F(a^I).
\]
Let us check that $G$ is a separately additive function. Due to symmetry, it suffices to prove additivity in the first variable $a_1$. Fix $a_2,\dots,a_n$. For any subset $J\subset \{2,\dots,n\}$, denote
\[
    b^J=(a^J_j)_{j=2}^n\quad \text{where }a_j^J=\left\{\begin{array}{ll}
        a_j & \text{if }j\in J \\
        0 & \text{if }j\not\in J.
    \end{array}  \right.
\]
With this notation, we can rewrite $G$ as
\[
    G(a_1,\dots,a_n)=\sum_{J\subset\{2,\dots,n\}} (-1)^{|J|}F(0,b^J) +  \sum_{J\subset\{2,\dots,n\}} (-1)^{|J|+1}F(a_1,b^J).
\]
To verify additivity in $a_1$, consider $G(a_1+b_1,a_2,\dots,a_n)$:
\begin{align*}
    G(a_1+b_1,\dots,a_n)&=\!\sum_{J\subset\{2,\dots,n\}}\!(-1)^{|J|}F(0,b^J)\! +\!  \sum_{J\subset\{2,\dots,n\}}\! (-1)^{|J|+1}F(a_1+b_1,b^J)\\
    &=\sum_{J\subset\{2,\dots,n\}} (-1)^{|J|}F(0,b^J) + \\
    &\quad +\sum_{J\subset\{2,\dots,n\}} (-1)^{|J|+1}\left(F(a_1,b^J)+F(b_1,b^J)-F(0,b^J)\right)\\
    &=G(a_1,\dots,a_n)+G(b_1,\dots,a_n).
\end{align*}

Thus, $G$ is separately additive. Now, we get that
\[
    F(a_1,\dots,a_n)=(-1)^nG(a_1,\dots,a_n)-\sum_{\substack{
         I\subset\{1,\dots,n\}  \\
           I\neq\{1,\dots,n\}}    }(-1)^{|I|+n}F(a^I),
\]
which is the representation we were looking for, since $F(a^I)$ are also separately affine additive functions.
\end{proof}

\begin{prop}\label{teo:descompadit}
A mapping $F:\mathbb{R}_+^n\to\mathbb{R}$ is a separately affine additive function if and only if for every subset $I\subset\{1,\dots,n\}$ there is a separately additive function $G_I:\mathbb{R}_+^{|I|}\to\mathbb{R}$ (where $G_\emptyset$ is a constant) such that
\[
    F(a_1,\dots,a_n)=\sum_{\begin{array}{c}
         I\subset \{1,\dots,n\}  \\
    \end{array}} G_I((a_i)_{i\in I}).
\]
Moreover, this decomposition is unique, that is, given two possible decompositions $\{G_I\}$ and $\{G'_{I}\}$ for $F$, one has that $G_I=G_I'$ for any $I\subset\{1,\dots,n\}$.
\end{prop}
\begin{proof}

For the direct implication, we will prove it by induction on the dimension. For $n=1$, let $F:\mathbb{R}_+\to\mathbb{R}$ be an affine additive function. By Lemma \ref{lema:descomp1},
there exists an additive function $G:\mathbb{R}_+\to\mathbb{R}$ and a constant $C\in\mathbb{R}$ such that
\[
    F(x)=G(x)+C,
\]
which is the representation we are looking for.

Assume the statement holds for dimension $n-1$. Let $F:\mathbb{R}^n_+\to\mathbb{R}$ be defined as a separately affine additive function. By Lemma \ref{lema:descomp1}, there exist $G$ separately additive function and $F_I$ separately affine additive functions, such that
\[
    F(a_1,\dots,a_n)=G((a_i)_{i\in\{1,\ldots,n\}})+\sum_{\substack{
         I\subset\{1,\dots,n\}  \\
           I\neq\{1,\dots,n\}}} F_I((a_i)_{i\in I}).
\]

Applying the induction hypothesis to each $F_I$, we deduce that every $F_I$ is a sum of separately additive functions. Consequently, $F$ itself is a sum of separately additive functions of the different dimensions. Since the sum of separately additive functions is again separately additive, this completes the induction.

Conversely, for any $I\subset\{1,\dots,n\}$, let $G_I:\mathbb{R}^{|I|}_+\to\mathbb{R}$ be separately additive functions (with $G_\emptyset$ a constant). Define
\[
    F(a_1\dots,a_n) = \sum_{\begin{array}{c}
         I\subset \{1,\dots,n\}  \\
    \end{array}} G_I((a_i)_{i\in I}),
\]
and let us check that $F$ is separately affine additive. Due to symmetry, we only need to prove that it is affine additive for the first variable $a_1$. Fix $a_2,\dots,a_n$ and, for any subset $J\subset \{2,\dots,n\}$, we denote $\alpha^J=(a_i)_{i\in J}$. 
Computing explicitly,
\begin{align*}
    F(a_1&+b_1,a_2,\dots,a_n) =\sum_{J\subset \{2,\dots,n\} }G_{\{1\}\cup J}(a_1+b_1,\alpha^J)+\sum_{J\subset\{2,\dots,n\}}G_J(\alpha^J)\\
    &= \sum_{J\subset \{2,\dots,n\} }G_{\{1\}\cup J}(a_1,\alpha^J)+\sum_{J\subset \{2,\dots,n\} }G_{\{1\}\cup J}(b_1,\alpha^J)   +\sum_{J\subset\{2,\dots,n\}}G_J(\alpha^J)\\
    &= \sum_{J\subset \{2,\dots,n\} }G_{\{1\}\cup J}(a_1,\alpha^J)+\sum_{J\subset\{2,\dots,n\}}G_J(\alpha^J)+\\
    &\,\,\,\,\,\,\,+ \sum_{J\subset \{2,\dots,n\} }G_{\{1\}\cup J}(b_1,\alpha^J)   +\sum_{J\subset\{2,\dots,n\}}G_J(\alpha^J)-  \sum_{J\subset\{2,\dots,n\}}G_J(\alpha^J)\\
    &=F(a_1,\dots,a_n)+F(b_1,\dots,a_n)-\sum_{J\subset\{2,\dots,n\}}G_J(\alpha^J).
\end{align*}
Now, notice that 
\begin{align*}
    F(0,a_2,\dots,a_n)& =\sum_{J\subset \{2,\dots,n\} }G_{\{1\}\cup J}(0,\alpha^J) +\sum_{J\subset\{2,\dots,n\}}G_J(\alpha^J)\\
    &=\sum_{J\subset\{2,\dots,n\}}G_J(\alpha^J),
\end{align*}
since for any separately additive function, one has that
\[
    G_{\{1\}\cup J}(0,\alpha^J)=G_{\{1\}\cup J}(0+0,\alpha^J)=G_{\{1\}\cup J}(0,\alpha^J)+G_{\{1\}\cup J}(0,\alpha^J).
\]

Finally, let us check uniqueness: Let $\{G_I\}$ and $\{G_I'\}$ be two decompositions of $F$ such that
\[
    \sum_{I\subset\{1,\dots,n\}}G_I((a_i)_{i\in I})=F(a_1,\dots,a_n)=\sum_{I\subset\{1,\dots,n\}}G_I'((a_i)_{i\in I})
\]
for any $a_i\geq 0$. We will prove the equality $G_I=G_I'$ by induction on the number of elements of $I$. Taking $a_i=0$ for all $i$, 
\[
    G_\emptyset=F(0,\dots,0)=G'_\emptyset.
\]
Now, assume that $G_I=G_I'$ for any $I\subset\{1,\dots,n\}$ such that $|I|=m-1$ with $m\leq n$ and let us prove that $G_I=G_I'$ if $|I|=m$. We may assume with no loss of generality that $I=\{1,\dots,m\}$. We have that
\[
    \sum_{J\subset I}G_J((a_i)_{i\in J})=F(a_1,\dots,a_m,0,\dots,0)=\sum_{J\subset I}G'_J((a_i)_{i\in J}).
\]
Isolating the term corresponding to $J=I$ on both sides
\[
    G_I((a_i)_{i\in I})+\sum_{\substack{J\subset I\\ J\neq I}}G_J((a_i)_{i\in J})=G_I'((a_i)_{i\in I})+\sum_{\substack{J\subset I\\ J\neq I}}G_J'((a_i)_{i\in J}),
\]
and, by the induction hypothesis, for any $J\subset I$ with $J\neq I$ we get that $G_J=G_J'$. Thus,
\[
    G_I((a_i)_{i\in I})=G_I'((a_i)_{i\in I})
\]
for any $a_i\geq 0$, as claimed.
\end{proof}

Combining Lemma \ref{lema:ident.tras} and Proposition \ref{teo:descompadit} we obtain the following characterization of translation invariant valuations on $\Pa(\{e_i\}_{i=1}^n)$.

\begin{teorema}\label{t:McMullen-pa}
If a mapping $V:\Pa(\{e_i\}_{i=1}^n)\to\mathbb{R}$ a translation invariant valuation, then for any $I\subset\{1,\dots,n\}$ there is a separately additive function $G_I:\mathbb{R}_+^{|I|}\to\mathbb{R}$ such that
\[
    V\left(\sum_{i=1}^n[0,a_i]e_i\right)=\sum_{I\subset\{1,\dots,n\}}G_I((a_i)_{i\in I}).
\]
Conversely, given a family of separately additive functions $\{G_I\}$ for any $I\subset\{1,\dots,n\}$ such that $G_I:\mathbb{R}^{|I|}_+\to\mathbb{R}$, then the function $V:\Pa(\{e_i\}_{i=1}^n)\to\mathbb{R}$ defined as
\[
    V\left(\sum_{i=1}^n [a_i,b_i]e_i\right)= \sum_{I\subset\{1,\dots,n\}}G_I\left((b_i-a_i)_{i\in I}\right)
\]
is a translation invariant valuation. Moreover, the decomposition as a sum of separately additive functions is unique.
\end{teorema}

\begin{rem}\label{rem:decomposition continuous}
Note that, following the process given in the uniqueness proof of Proposition \ref{teo:descompadit}, for any $I\subset\{1,\dots,n\}$ one can write the $G_I((a_{i})_{i\in I})$ as the sums and subtractions of terms with $V$. Simply note that
\[
    G_\emptyset=V(\{0\})
\]
and for any $I$ such that $|I|>0$ one has that
\[
    G_I((a_i)_{i\in I})=V\left(\sum_{i\in I}[0,a_i]e_i\right)-\sum_{\substack{J\subset I\\J\neq I}}G_J((a_i)_{i\in J}).
\]
Applying induction on the size of $I$, we get that $G_I$ can be written as the sum and subtraction of terms that involve $V$ over sums of intervals. 
From this, it is clear to see that $V$ is continuous on $\Pa(\{e_i\}_{i=1}^n)$ if and only if all $G_I$ are continuous on $\mathbb{R}^{|I|}_+$. In particular, it is enough for them to be separately continuous since
\[
    \sum_{i=1}^n[0,a_i(j)]e_i\xrightarrow[j\to\infty]{}\sum_{i=1}^n[0,b_i]e_i
\]
in $\Pa(\{e_i\}_{i=1}^n)$ if and only if $a_i(j)\xrightarrow[j\to\infty]{}b_i$ in $\mathbb{R}_+$ for all $i$.
\end{rem}

Theorem \ref{t:McMullen-pa} can be interpreted as a version of McMullen's Theorem on decomposition for valuations on polytopes (see \cite{Alesker}). We characterize the homogeneous parts of the decomposition as separately additive functions of different dimensions, specifically when working within $\Pa(\{e_i\}_{i=1}^n)$. This formulation serves as a substitute for McMullen's Theorem and we will use Theorem \ref{t:McMullen-pa} in the proof of Theorem \ref{t:main}. This also leads to the following automatic continuity result. 

\begin{coro}\label{c:automatic continuity}
Let $V:\Pa(\{e_i\}_{i=1}^n)\to\mathbb{R}$ be a translation invariant valuation. If $V$ is bounded on bounded sets (that is $\sup_{K\subset B}|V(K)|<\infty$), then $V$ is continuous. 
\end{coro}
\begin{proof}
By Theorem \ref{t:McMullen-pa}, there exist $\{G_I\}_{I\subset\{1,\dots,n\}}$ separately additive functions such that
\[
    V\left(\sum_{i=1}^n[0,a_i]e_i\right)=\sum_{I\subset\{1,\dots,n\}}G_I\left((a_i)_{i\in I}\right),
\]
for any $a_i\geq 0$. First, we show that the fact that $V$ is bounded on bounded sets implies that all $G_I$ are continuous. Clearly, $G_\emptyset$ is continuous, since it is a constant, so we will proceed by induction on the size of $I$. Assume that for any $I\subset\{1,\dots,n\}$ such that $|I|\leq m-1$ we have proved that $G_I$ is continuous. For any subset $I\subset\{1,\dots,n\}$ with $|I|=m$, one has
\[
    V\left(\sum_{i\in I}[0,a_i]e_i\right)=\sum_{J\subset I}G_J((a_i)_{i\in J})=G_I((a_i)_{i\in I})+\sum_{\substack{J\subset I\\ J\neq I}}G_J((a_i)_{i\in J})
\]
for any $a_i\geq 0$. Since $V$ is bounded on bounded sets, and $G_J$ is continuous for every $J\subset I$ with $J\neq I$, it follows that $G_I$ has to be bounded on the positive part of the ball $B_+$. Let us extend $G_I$ to the whole $\mathbb{R}^{|I|}$ as an odd function $\hat{G}_I$, that is
\[
    \hat{G}_I((\varepsilon_i a_i)_{i\in I})=\left(\prod_{i\in I}\varepsilon_i\right)G_I((a_i)_{i\in I})
\]
where $\varepsilon_i=\pm1$. It is straightforward to check that $\hat{G}_I$ is separately additive on $\mathbb{R}^{|I|}$, and bounded on $B$. By Lemma \ref{lema:multiseparada} we get that $\hat{G}_I$ is separately linear which implies that it is continuous. Since $G_I$ is merely the restriction of $\hat{G}_I$, we have that $G_I$ is also continuous. As we mentioned in Remark \ref{rem:decomposition continuous}, the fact that all $G_I$ are continuous implies that $V$ is continuous. 
\end{proof}

Note that, in the previous result, the role of $B$ can be exchanged for any set with nonempty interior and, similarly, the Borel condition yields an analogous corollary. In particular, we obtain the following result, which mirrors Theorem \ref{teo:CFE}. 


\begin{teorema}\label{teo:valuaciones_continuidad}[Automatic continuity]
Let $V:\Pa(\{e_i\}_{i=1}^n)\to\mathbb{R}$ be a translation invariant valuation. The following are equivalent:
\begin{itemize}
    \item[(1)] $V$ is continuous
    \item[(2)] $V$ is bounded on all parallelotopes inside of a set with nonempty interior.
    \item[(3)] $V$ is Borel.
\end{itemize}
\end{teorema}

In this Theorem, the translation invariant condition is necessary. Note that a complete description of general (not necessarily translational invariant) valuations on $\Pa(\{e_i\}_{i=1}^n)$ can be given. For completeness, we include the details in Appendix \ref{S:AppendixA}. As a consequence, one easily observes that no automatic continuity type result can be expected in this more general setting.

The characterization given by Theorem \ref{t:McMullen-pa}, allows a very simple proof of the following well-known result. 
\begin{lema}\label{lema:Simple}
Let $V:\mathcal{K}^n\to\mathbb{R}$ be a translation invariant valuation that is $m$-homogeneous with $m\not\in\{0,1,\dots,n\}$. Then, for any basis $\{e_i\}_{i=1}^n$ of $\mathbb{R}^n$, one has $V|_{\Pa(\{e_i\}_{i=1}^n)}=0$.

\end{lema}
\begin{proof}
Let $\{e_i\}_{i=1}^n$ be a basis of $\mathbb{R}^n$. By Theorem \ref{t:McMullen-pa}, we get that
\[
    V\left(\sum_{i=1}^n[0,a_i]e_i\right)=\sum_{I\subset\{1,\dots,n\}}G_I((a_i)_{i\in I}).
\]
for any $a_i\geq 0$. Since $V$ is $m$-homogeneous, we get that for any $\lambda\in\mathbb Q_+$
\begin{align*}
    \lambda ^mV\left(\sum_{i=1}^n[0,a_i]e_i\right)&= V\left(\sum_{i=1}^n[0,\lambda a_i]e_i\right)=\sum_{I\subset\{1,\dots,n\}}G_I((\lambda a_i)_{i\in I})\\
    &=\sum_{I\subset\{1,\dots,n\}}\lambda^{|I|}G_I(( a_i)_{i\in I}),    
\end{align*}
where we used the fact that additive functions are $\mathbb{Q}$-linear \cite[Theorem 5.2.1]{Kuczma}

Comparing both sides of the equation, we find that the mapping $\lambda\mapsto\lambda^m$ must coincide on $\mathbb{Q}_+$ with a polynomial in $\lambda$ of degree at most $n$. This can not hold when $m\not\in\{0,1,\dots,n\}$. Therefore,
\[
    V\left(\sum_{i=1}^n[0,a_i]e_i\right)=0.
\]
\end{proof}

\section{A new characterization of the volume in \texorpdfstring{$\mathbb{R}^n$}{Rn}}\label{S:main}
The previous section enables us to present a generalization of one of Hadwiger's characterizations of the volume. To begin with, we introduce a small technical lemma that allows us to avoid the usual Inclusion-Exclusion property (which is not a priori guaranteed for non-continuous valuations). These lemmas will be used to decompose the convex body into smaller pieces determined by a grid. Note that these lemmas actually work for weakly additive maps since all the cutting is done with hyperplanes.

Let $u\in\mathbb{S}^{n-1}$ and $r< R$ be real numbers, we define the following sets:
\[
    \begin{array}{ll}
        H_{[r,R]}^u\coloneqq\{x\in\mathbb{R}^{n}:\langle x,u\rangle\in[r,R]\}, & H_r^u\coloneqq\{x\in\mathbb{R}^{n}:\langle x,u\rangle=r\} \\
        H_{[r,\infty]}^u\coloneqq\{x\in\mathbb{R}^{n}:\langle x,u\rangle\geq r\}, & H_{[-\infty,r]}^u\coloneqq \{x\in\mathbb{R}^{n}:\langle x,u\rangle\leq r\}.
    \end{array}
\]
To simplify notation throughout this section, we adopt the convention that $V(\emptyset)=0$ for any valuation $V$ and $\left\lceil x\right\rceil$ denotes the least integer satisfying $x\leq\left\lceil x\right\rceil$. 

\begin{lema}
Let $V:\mathcal{K}^n\to\mathbb{R}$ be a valuation, $u\in\mathbb{S}^{n-1}$ and $K\in\mathcal{K}^n$ with $K\subset MB$ ($B$ is the euclidean unit ball) for some $M>0$. Then, for any $\varepsilon<M$ we get 
\[
    V(K)=\sum_{l=0}^{\left\lceil \frac{2M}{\varepsilon}\right\rceil-1}V(K\cap H^u_{[-M+l\varepsilon,-M+(l+1)\varepsilon]}) -\sum_{l=1}^{\left\lceil \frac{2M}{\varepsilon}\right\rceil-1}V(K\cap H^u_{-M+l\varepsilon}).
\]
\end{lema}

\begin{proof}
The first step is to take $K\subset MB$ and cut it with the hyperplane $H_{-M+\varepsilon}^u$ and by weak additivity we get that 
\[
    V(K)=V(K\cap H_{[-\infty,-M+\varepsilon]}^u )+V(K\cap H_{[-M+\varepsilon,\infty]}^u )-V(K\cap H^u_{-M+\varepsilon} ).
\]
Since $K\subset MB$ we have that $K\subset H_{[-M,M]}^u$, so in particular
\[
    V(K)=V(K\cap H_{[-M,-M+\varepsilon]}^u)+V(K\cap H_{[-M+\varepsilon,M]}^u )-V(K\cap H^u_{-M+\varepsilon} ).
\]
Now, work with $K\cap H^u_{[-M+\varepsilon,M]}$ and cut with the hyperplane $H^u_{-M+2\varepsilon}$. By weak additivity, we have
\begin{align*}
    V(K\cap H^u_{[-M+\varepsilon,M]})&=V(K\cap  H^u_{[-M+2\varepsilon,M]})+V(K\cap H^u_{[-M+\varepsilon,-M+2\varepsilon]})\\
    &\quad -V(K\cap  H^u_{-M+2\varepsilon}).
\end{align*}

Repeating this process $\left\lceil \frac{2M}{\varepsilon}\right\rceil$ times we get the result. 
\end{proof}

By iterating the previous lemma on different directions, given by any basis $\{e_i\}_{i=1}^n\subset\mathbb{S}^{n-1}$, for any $\varepsilon>0$, we can subdivide any convex set $K$ with a grid aligned with the coordinate axes with respect to the basis. This allows us to compute the value of $V(K)$ as the sum of different pieces fitting inside the grid. Some of the pieces have dimension strictly less than $n$ and will be discarded in the proof of Lemma \ref{l:counting} below, because we will show that the valuation is simple, that is, $V(L)=0$ for any convex body $L$ with $\dim L\leq n-1$.

\begin{lema}\label{l:grid decomposition}
Let $V:\mathcal{K}^n\to\mathbb{R}$ be a valuation, let $\{e_i\}_{i=1}^n\subset\mathbb{S}^{n-1}$ be a basis of $\mathbb{R}^n$ and let $K\in\mathcal{K}^n$ be contained in $MB$ for some $M>0$. Then, for any $\varepsilon<M$ we have
\[
    V(K)=\sum_{i_1,\dots,i_n=1}^{\left\lceil\frac{2M}{\varepsilon}\right\rceil-1}V\left(K\cap \bigcap_{j=1}^nH^{e_j}_{[-M+i_j\varepsilon,-M+(i_j+1)\varepsilon]}\right)+\sum_{l}\theta_l V(K_l)
\]
where the last term is a finite sum, $\theta_l\in\{1,-1\}$ and each $K_l\in\mathcal{K}^n$ satisfies $\dim K_l\leq n-1$.
\end{lema}

The following lemma provides the missing key component of the proof of Theorem \ref{t:main}. It is similar to a volume counting argument, but applied on the values of the valuation on the sets defined by the grid above.  

For convenience, given a valuation $V:\mathcal{K}^n\to\mathbb{R}$ which is bounded on $\mathcal{K}(B)$, let us denote $\|V\|=\sup_{K\subset B}|V(K)|$.

\begin{lema}\label{l:counting}
Let $V:\mathcal{K}^n\to\mathbb{R}$ be a translation invariant valuation that is $m$-homogeneous with $m> n-1$ and is bounded on $\mathcal{K}(B)$. If for some basis $\{e_i\}_{i=1}^n$ of $\mathbb{R}^n$, $V|_{\Pa(\{e_i\}_{i=1}^n)}=0$, then $V=0$. 
\end{lema} 

\begin{proof}
Let us proceed by induction on the dimension. If $n=1$, then $\mathcal{K}^1=\Pa(\{e_1\})$ for any nonzero vector $e_1$ and we have the result. Assume that it is true for $n-1$. First, let us check that $V$ is simple. Let $E$ be a subspace such that $\dim E=n-1$. Then $V|_E$ is a translation invariant valuation that is $m$-homogeneous. By Lemma \ref{lema:Simple}, we get that $V|_E$ is null for any parallelotope with regards to any basis. Applying the induction hypothesis (since $m>n-1>n-2$), we get that $V|_E=0$.


From this point on, we assume that $\{e_i\}_{i=1}^n$ is an orthonormal basis of $\mathbb{R}^n$ for which $V|_{\Pa(\{e_i\}_{i=1}^n)}=0$. In the general case, the ideas and the structure of the proof remain the same, but the computations become slightly more involved. Therefore, to keep the presentation as simple as possible, we work under the assumptions of orthogonality and unit norm.

Let $K\in\mathcal{K}^n$ of full dimension. By the translation invariant property we may assume that there are $r,R>0$ such that 
\[
    B(0,r)\subset K\subset B(0,R)
\]
where $B(0,\alpha)$ is the closed Euclidean ball centered at the origin with radius $\alpha>0$. We will also denote by $B_\infty(x,\alpha)=\{y\in\mathbb{R}^n: \|x-y\|_\infty\leq\alpha \}$ for any $\alpha>0$, the box centered at $x$ and with side length equal to $2\alpha$, where the norm $\|\cdot\|_\infty$ is taken with respect to the orthonormal basis $\{e_i\}_{i=1}^n$. Let $\varepsilon>0$. For $0\leq i_1,\ldots,i_n \leq \left\lceil\frac{2R}{\varepsilon}\right\rceil-1$, define the points $x_{i_1,\dots,i_n}$ as the center of the following box
\[
    B\left(x_{i_1,\dots,i_n},\frac{\varepsilon}{2}\right)=\bigcap_{j=1}^nH^{e_j}_{[-R+i_j\varepsilon,-R+(i_j+1)\varepsilon]}.
\]
Applying Lemma \ref{l:grid decomposition} with the basis $\{e_i\}_{i=1}^n$ and using the fact that $V$ is simple, we get that
\[
    V(K)=\sum_{i_1,\dots,i_n=1}^{\left\lceil\frac{2R}{\varepsilon}\right\rceil-1}V\left(K\cap B\left(x_{i_{1},\dots,i_n},\frac{\varepsilon}{2}\right)\right).
\]
Note that if $K\cap B\left(x_{i_{1},\dots,i_n},\frac{\varepsilon}{2}\right)=B\left(x_{i_{1},\dots,i_n},\frac{\varepsilon}{2}\right)$, then $V(K\cap B\left(x_{i_{1},\dots,i_n},\frac{\varepsilon}{2}\right))=0$, since $V|_{\Pa(\{e_i\}_{i=1}^n)}=0$. 
Let us define the index set
\[
    J=\left\{(i_1,\dots,i_n) :K\cap B\left(x_{i_{1},\dots,i_n},\frac{\varepsilon}{2}\right)\neq B\left(x_{i_{1},\dots,i_n},\frac{\varepsilon}{2}\right)\right\}.
\]
It is clear that if $I\in J$, then
\[
    B_\infty \left(x_{I},\frac{\varepsilon}{2}\right)\cap (\mathbb{R}^n\setminus K)\neq \emptyset.
\]
Therefore, for any $I\in J$
\[
    B_\infty\left(x_I,\frac{\varepsilon}{2}\right)\subset \partial K+B_\infty(0,\varepsilon)\subset \partial K+B(0,\sqrt{n}\varepsilon)
\]
where we used the fact that $B_\infty(x,s)\subset B(x,\sqrt{n}s)$ for all $x\in\mathbb{R}^n$ and $s>0$. Note here that if we were not using an orthonormal basis, the only difference is that $\sqrt{n}$ would be replaced by a certain constant depending only on the basis.
Hence, we have
\[
    \bigcup_{I\in J}B_\infty\left(x_i,\frac{\varepsilon}{2}\right)\subset \partial K+B(0,\sqrt{n}\varepsilon).
\]
Note that 
\[
    \partial K+B(0,\sqrt{n}\varepsilon)=\{x:d(x,K)\leq \sqrt{n}\varepsilon\}\setminus \{x: d(x,\mathbb{R}^n\setminus K)\geq \sqrt{n}\varepsilon\}
\]
where $d$ denotes the Euclidean metric in $\mathbb{R}^n$. The first set is clearly $K+\sqrt{n}\varepsilon B$ and the second term can be rewritten as
\[
    \{x: d(x,\mathbb{R}^n\setminus K)\geq \sqrt{n}\varepsilon\}=\{x: x+\sqrt{n}\varepsilon B\subset K\}=\{x:B(x,\sqrt{n}\varepsilon)\subset K\}.
\]
Using the operation of Minkowski substraction of convex sets denoted with $\sim$, we can rewrite
\[
    \partial K+B(0,\sqrt{n}\varepsilon)= (K+\sqrt{n}\varepsilon B)\setminus (K\sim \sqrt{n} \varepsilon B).
\]
Thus, we get
\[
    \bigcup_{I\in J}B_\infty\left(x_I,\frac{\varepsilon}{2}\right)\subset \left(K+\sqrt{n}{\varepsilon}B\right)\setminus \left(K\sim \sqrt{n} {\varepsilon}B \right).
\]
Since the volume is monotone, the intersection of the different boxes has dimension smaller than or equal to $n-1$ and $K\sim\sqrt{n}\varepsilon B\subset K+\sqrt{n}\varepsilon B$, we get that
\[
    |J| \varepsilon^n\leq \text{Vol}\left(K+\sqrt{n}{\varepsilon}B\right)-\text{Vol}\left(K\sim\sqrt{n}{\varepsilon}B\right).
\]

On the one hand, it is well known that 
\[
    \frac{\text{Vol}\left(K+\sqrt{n}{\varepsilon}B\right)-\text{Vol}(K)}{\sqrt{n}{\varepsilon}}\xrightarrow[\varepsilon\to0^+]{}S(K),
\]
where $S(K)$ denotes the surface area of $K$ (see \cite[Chapter 4]{Schneider}). On the other hand, since $B(0,r)\subset K\subset B(0,R)$, for $\varepsilon$ small enough we get
\[
    \frac{\text{Vol}(K)-\text{Vol}\left(K\sim \sqrt{n}{\varepsilon}B\right)}{\sqrt{n}{\varepsilon}}\leq 4nr^{-2}RV_n(K),
\]
(see \cite[Lemma 2.3.6]{Schneider}). Combining these results we get
\begin{align*}
    |J|&\leq\! \left[\frac{\text{Vol}\left(K+\sqrt{n}{\varepsilon}B\right)-\text{Vol}(K)}{\sqrt{n}\varepsilon}+\frac{\text{Vol}(K)-\text{Vol}\left(K\sim \sqrt{n}{\varepsilon}B\right)}{\sqrt{n}{\varepsilon}}   \right]{\sqrt{n}}\left(\frac{1}{\varepsilon}\right)^{n-1}\\
    &\leq C_{K,n}\left(\frac{1}{\varepsilon}\right)^{n-1},
\end{align*}
where $C_{K,n}$ is a constant that depends on $K$ and on the dimension $n$ (and on the basis $(e_i)_{i=1}^n$). Finally, note that for $I\in J$ we have
\[
    K\cap B_\infty\left(x_I,\frac{\varepsilon}{2}\right)\subset B_\infty\left(x_I,\frac{\varepsilon}{2}\right)\subset B \left(x_I,\sqrt{n}\frac{\varepsilon}{2}\right),
\]
and since $V$ is translation invariant and $m$-homogeneous it follows that
\[
    \left|V\left( K\cap B_\infty\left(x_I,\frac{\varepsilon}{2}\right)\right)\right|\leq \sup_{L\subset \sqrt{n}\frac{\varepsilon}{2}B}|V(L)| = \|V\| \left(\sqrt{n}\frac{\varepsilon}{2}\right)^m.
\]
Therefore, we have
\begin{align*}
    |V(K)|&=\left| \sum_{i_1,\dots,i_n=1}^{\left\lceil\frac{2R}{\varepsilon}\right\rceil-1}V\left(K\cap  B\left(x_{i_{1},\dots,i_n},\frac{\varepsilon}{2}\right)\right) \right|=
    \left| \sum_{I\in J}V\left(K\cap  B_\infty\left(x_I,\frac{\varepsilon}{2}\right)\right) \right|\leq \\
    &\leq |J| \|V\| \left(\sqrt{n}\frac{\varepsilon}{2}\right)^m \leq C'_{K,n} \|V\|\varepsilon^m \frac{1}{\varepsilon^{n-1}}=C_{K,n}'\|V\|  \varepsilon^{m-n+1}.
\end{align*}
for any side length $\varepsilon$ of the grid small enough and with $C'_{K,n}$ a positive constant independent of $\varepsilon$. In particular, taking $\varepsilon\xrightarrow[]{}0^+$ we get that $V(K)=0$. Since this was done for any full dimensional body and the valuation is simple, we get the result.
\end{proof}




\begin{proof}[Proof of Theorem \ref{t:main}]
Let $V$ satisfy the hypothesis and consider $V|_{\Pa(\{e_i\}_{i=1}^n)}$. This valuation is bounded and translation invariant, thus by Corollary \ref{c:automatic continuity}, it is continuous and the $G_I$ given by Theorem \ref{t:McMullen-pa} are continuous and separately linear. Only one of these is $n$-homogeneous, that is, $G_{\{1,\dots,n\}}$, and we get 
\[
    V|_{\Pa(\{e_i\}_{i=1}^n)}\left(\sum_{i=1}^n[0,a_i]e_i\right)=G_{\{1,\dots,n\}}(a_1,\dots,a_n)
\]
for any $a_i\geq 0$. Now since $G$ is separately linear we get that there exists $c\in\mathbb{R}$ such that $G_{\{1,\dots,n\}}(a_1,\dots,a_n)=c\prod_{i=1}^na_i$ which is a multiple of the volume $\text{Vol}$ and yields $V|_{\Pa(\{e_i\}_{i=1}^n)}=c\text{Vol}$. Lastly, define $\phi\coloneqq V-c\text{Vol}$. This is a valuation that satisfies the conditions of the Lemma \ref{l:counting}, and which leads to $\phi=0$.
\end{proof}











Again, we restate that in the previous Lemma \ref{l:counting} the valuation condition can be relaxed, we only need the mapping $V$ to be weakly additive. 
This property is strictly weaker than being a valuation (see \cite[Page 173]{Debiladitiva} for an example of a weakly additive function on $\mathcal{K}^n$ which is not a valuation). Nevertheless, if $V$ is weakly additive, it is straightforward to see that the restriction $V|_{\Pa(\{e_i\}_{i=1}^n)}$ is a valuation on $\Pa(\{e_i\}_{i=1}^n)$, allowing us to apply the results from Section \ref{S:par} 

Another observation is that the mapping $V$ does need to be $\mathbb{R}$-homogeneous, it is enough for it to be $\mathbb{Q}$-homogeneous. This follows from the fact that, in the proof of Lemma \ref{l:counting}, we can choose $\varepsilon$ such that $\sqrt{n}\frac{\varepsilon}{2}$ is rational. We summarize these in the following result.

\begin{coro}
Let $V:\mathcal{K}^n\to\mathbb{R}$ be weakly additive, translation invariant, bounded on $\mathcal{K}(B)$ and $n$-homogeneous with respect to $\mathbb Q$. Then, $V$ is proportional to the Lebesgue measure.
\end{coro}

\section{Degrees of homogeneity for a bounded valuation}\label{S:homogeneidad}

In the previous section, we used Lemma \ref{l:counting} to derive a stronger version of Hadwiger's characterization of the volume. In this section, we exploit this lemma further to extract some results about the homogeneity of bounded valuations. Specifically, we show that for bounded, translation invariant valuations on $\mathcal{K}^n$ the set of admissible degrees of homogeneities is exactly $[0,n-1]\cup\{n\}$. 

\begin{prop}\label{prop:homogeneidades}
Let $V:\mathcal{K}^n\to\mathbb{R}$ be a translation invariant valuation such that $\sup_{K\subset B}|V(K)|<\infty$. If $V$ is $m$-homogeneous, then $m\in[0,n-1]\cup \{n\}$.
\end{prop}
\begin{proof}
First, since $V$ is bounded, we must have $m\geq 0$; otherwise taking $K\in\mathcal{K}(B)$ with $V(K)\neq 0$ would yield
\[
    |V(\lambda K)|=\lambda^m|V(K)|
\]
which tends to infinity as $\lambda$ tends to 0, even though $\lambda K\subset B$.\\
Next, if $m>n-1$ and $m\neq n$, then combining Lemma \ref{lema:Simple} and Lemma \ref{l:counting} shows that $V=0$.
\end{proof}

To see that this set cannot be reduced any further, we construct bounded and translation invariant valuations, $\phi_l$, that are $m$-homogeneous for any $m\in[0,n-1]$. These examples show both that each degree in this interval actually occurs, and that the argument from Lemma \ref{l:counting} can not be pushed to extend the values of $m$. Indeed, each $\phi_l$ vanishes on polytopes while satisfying $\phi_l(B)>0$. These valuations are analogous to the $p$-affine surface areas (see \cite[Section 3.3]{MonikayFabian} for a gentle introduction) and are based on the Gauss Kronecker curvature. Such valuations have been previously introduced in \cite{Curvature-Integrals} and, in the case $n=2$, have been characterized as the upper semicontinuous valuations that are rotation invariant and homogeneous of degree different from 0, 1 or 2 (see \cite{Planar}).

\begin{defin}
Given $K\in\mathcal{K}^n$, let $\kappa(K,x)$ denote the Gauss-Kronecker curvature of $\partial K$ at the point $x$.
\end{defin}
A well known result due to Aleksandrov (see \cite{Alexandroff} or \cite[Theorem 2.6.1]{Schneider}) states that the boundary of a convex body is twice differentiable almost everywhere with respect to the $(n-1)$-dimensional Hausdorff measure. This implies that the curvature is well defined almost everywhere, and moreover, $x\mapsto \kappa(K,x)$ is measurable. This mapping can also be described as the density of the absolutely continuous part of the curvature measure $C_0(K,\cdot)$ with respect to the $(n-1)$-dimensional Hausdorff measure. We refer the reader to \cite[Chapter 4]{Schneider} and \cite[Theorem 3.2]{Hug} for the definition of $C_0(K,\cdot)$ and details. As a consequence, we have the following.

\begin{lema}\label{lema:homogeneidad_curvatura}
The mapping $\kappa$ is homogeneous, in particular, for $\lambda> 0$ we get that
\[
    \kappa(\lambda K,\lambda x)=\frac{1}{\lambda^{n-1}}\kappa(K,x)
\]
\end{lema}
\begin{proof}
This is straightforward using the fact that $\kappa$ coincides with the density of the absolutely continuous part of the curvature measure $C_0(K,\cdot)$, since the Lebesgue decomposition of the measure $C_0(K,\cdot)$ is unique and it satisfies that $C_0(\lambda K,\lambda \beta)=C_0(K,\beta)$ for any Borel set $\beta$ in $\mathbb{R}^n$. More explicitly, let $H^{n-1}|_{\partial K}(E)=H^{n-1}(\partial K\cap E)$. For any $K$, we have
\[
    C_0(K,\beta)=\int_\beta \kappa(K,x)\mathrm{d}H^{n-1}|_{\partial K}(x).
\]
Let $\lambda >0$. Since $C_0(K,\beta)=C_0(\lambda K,\lambda \beta)$, we obtain
\[
    \int_\beta \kappa(K,x)\mathrm{d}H^{n-1}|_{\partial K}(x)=\int_{\lambda\beta}\kappa (\lambda K,x)\mathrm{d}H^{n-1}|_{\partial(\lambda K)}(x).
\]
by the homogeneity of $C_0$. Now, define $f_\lambda:\mathbb{R}^n\to\mathbb{R}^n$ as $f_\lambda(x)=\lambda x$ and consider the push-forward measure $\nu$ given by $f_\lambda$ and the measure $H^{n-1}|_{\partial K}$. Then, for any Borel set $\beta$ of $\mathbb{R}^n$ one has
\begin{align*}
    \nu(\beta)&=H^{n-1}|_{\partial K}(f_\lambda^{-1}(\beta))=H^{n-1}(\partial K\cap \lambda^{-1}\beta)\\
    &=\frac{1}{\lambda^{n-1}}H^{n-1}(\partial(\lambda K)\cap \beta)=\frac{1}{\lambda^{n-1}}H^{n-1}|_{\partial(\lambda K)}(\beta).
\end{align*}
where we used the fact that $\partial (\lambda K)=\lambda\partial K$ and the homogeneity of the Hausdorff measure. That is, the push-forward measure is a multiple of the measure $H^{n-1}|_{\partial(\lambda K)}$. Thus,
\begin{align*}
    \int_{\lambda\beta}\kappa(\lambda K,x)\mathrm{d}H^{n-1}|_{\partial(\lambda K)}(x)& =\lambda^{n-1}\int_{\lambda\beta}\kappa(\lambda K,x)\mathrm{d}\nu(x)\\ 
    & =\lambda^{n-1}\int_{\mathbb{R}^n}\kappa(\lambda K,x)\chi_{\lambda \beta}(x)\mathrm{d}\nu,
\end{align*}
where $\chi_A(x)$ denotes the indicator function of the set $A$. Applying the property \cite[Theorem 3.6.1]{pushforward} we get that
\begin{align*}
    \int_{\mathbb{R}^n}\kappa(\lambda K,x)\chi_{\lambda \beta}(x)\mathrm{d}\nu& =\int_{\mathbb{R}^n}\kappa(\lambda K,f_\lambda(x))\chi_{\lambda \beta}(f_\lambda(x))\mathrm{d}H^{n-1}|_{\partial K}(x)\\
    &=\int_\beta \kappa(\lambda K,\lambda x) \mathrm{d}H^{n-1}|_{\partial K}(x).
\end{align*}
Thus, for any Borel set $\beta$,
\[
    \int_\beta \left(\kappa(K,x)-\lambda^{n-1}\kappa(\lambda K,\lambda K)\right)\mathrm{d}H^{n-1}|_{\partial K}(x)=0,
\]
which implies that
\[
    \kappa(K,x)=\lambda^{n-1}\kappa(\lambda K,\lambda x)
\]
almost everywhere with respect to $H^{n-1}|_{\partial K}$.
\end{proof}

\begin{defin}
For $l\in\mathbb{R}$, let us define $\phi_l:\mathcal{K}^n\to\mathbb{R}\cup\{\infty\}$ as
\[
    \phi_m(K)=\int_{\partial K}\kappa(K,x)^l\mathrm{d}H^{n-1}(x).
\]
\end{defin}

This is well defined since $\kappa(K,\cdot)$ is measurable and nonnegative, also, in the case $l=0$, we adopt the convention that $0^0=0$. We will need the following lemma taken directly from \cite[Lemma 5]{Schutt}.

\begin{lema}\label{l:Schutt}
Let $K,L\in\mathcal{K}^n$ such that $K\cup L\in\mathcal{K}^n$. For those $x\in\partial K\cap \partial L$ where all the curvatures $\kappa(K\cup L,\cdot)$, $\kappa(K\cap L,\cdot)$, $\kappa(K,\cdot)$ and $\kappa(L,\cdot)$ exist we have that
\[
    \begin{array}{c}
         \kappa(K\cup L,x)=\min\{\kappa(K,x),\kappa(L,x)\},  \\
         \kappa(K\cap L,x)=\max\{\kappa(K,x),\kappa(L,x)\}. 
    \end{array}
\]
\end{lema}

With this lemma and following the same procedure as in \cite{Schutt} we can prove that the previous mapping is a valuation.

\begin{lema}
$\phi_l$ is a valuation for any $l\in\mathbb R$.
\end{lema}
\begin{proof}
The key of this proof is based in the partition of the boundaries of the convex bodies. In particular, using the following identities:
\[
\begin{array}{l}
     \partial (K\cup L)=(\partial K\cap\partial L)\cup\left(\partial K\cap L^c\right)\cup(K^c\cap \partial L)\\
    \partial (K\cap L)=(\partial K\cap\partial L)\cup(\partial K\cap \text{Int}\,L)\cup(\text{Int}\,K\cap \partial L)\\
    \partial K=  (\partial K\cap\partial L)\cup (\partial K\cap L^c)\cup (\partial K\cap \text{Int}\,L)\\
    \partial L=  (\partial K\cap\partial L)\cup (\partial L\cap K^c)\cup (\partial L\cap \text{Int}\,K),
\end{array}
\]
where $K^c,L^c,\text{Int}\,K,\text{Int}\,L$ are the complementary of $K$, of $L$, the interior of $K$ and of $L$ respectively. In the following, we will assume that the integrals are all with respect to the Hausdorff measure $H^{n-1}$, so we will omit its symbol and the variable of integration. Consider
\begin{align*}
    \phi_l(K\cup L)=\int_{\partial K\cap\partial L}\kappa(K\cup L)^l+\int_{\partial K\cap L^c}\kappa(K\cup L)^l+\int_{\partial L\cap K^c}\kappa(K\cup L)^l,\\
    \phi_l(K\cap L)=\int_{\partial K\cap\partial L}\kappa(K\cap L)^l+\int_{\partial K\cap \text{Int}\, L}\kappa(K\cap L)^l+\int_{\partial L\cap \text{Int}\, K}\kappa(K\cap L)^l.
\end{align*}
    
Since $L^c, K^c, \text{Int}\, K,\text{Int}\, L$ are open sets and the curvature is a local invariant, we get that
\begin{align*}
    \phi_l(K\cup L)=\int_{\partial K\cap\partial L}\kappa(K\cup L)^l+\int_{\partial K\cap L^c}\kappa(K)^l+\int_{\partial L\cap K^c}\kappa(L)^l\\
    \phi_l(K\cap L)=\int_{\partial K\cap\partial L}\kappa(K\cap L)^l+\int_{\partial K\cap \text{Int}\, L}\kappa(K)^l+\int_{\partial L\cap \text{Int}\, K}\kappa(L)^l.
\end{align*}
Now, using Lemma \ref{l:Schutt} we get that 
\[
    \int_{\partial K\cap\partial L}\kappa(K\cup L)^l+\int_{\partial K\cap\partial L}\kappa(K\cap L)^l=\int_{\partial K\cap\partial L}\kappa(K)^l+\int_{\partial K\cap\partial L}\kappa(L)^l
\]
and reorganizing the terms we get the identity 
\[
    \phi_l(K\cup L)+\phi_l(K\cap L)=\phi_l(K)+\phi_l(L).
\]
\end{proof}
Since the curvature is translation invariant, that is, $\kappa(K+y,x+y)=\kappa(K,x)$, we get that the previous valuations are translation invariant. Not only that, they are also homogeneous.

\begin{lema}
The previous valuation $\phi_l$ is $(n-1)(1-l)$-homogeneous. 
\end{lema}
\begin{proof}
Using a similar procedure as in Lemma \ref{lema:homogeneidad_curvatura} we get
\begin{align*}
    \phi_l(\lambda K)&=\int_{\partial\lambda K}\kappa(\lambda K,x)^{l}\mathrm{d}H^{n-1}(x)=\int_{\partial K}\kappa (\lambda K,\lambda x)^{l}\lambda^{n-1}\mathrm{d}H^{n-1}(x)\\
    &=\phi_l(K)\lambda^{n-1-l(n-1)}
\end{align*}
\end{proof}

Now, if $l<0$, since $\kappa(P,x)=0$ for $P\in\mathcal{P}^n$ we get that $\phi_l(P)=\infty$ which makes the valuation not bounded. If $l>1$ the previous result implies that the valuations is homogeneous with negative degree, which implies that it is also not bounded on $\mathcal{K}(B)$. Thus we are only interested in the range $l\in [0,1]$. 

\begin{lema}
For $l\in[0,1]$ the previous valuation $\phi_l$ is bounded on $\mathcal{K}(B)$. 
\end{lema}
\begin{proof}
Since $\frac{1}{l}\in [1,\infty]$, by Hölder's inequality, we have
\[
    \phi_l(K)=\int_{\partial K}\kappa(K,x)^l\mathrm{d}H^{n-1}(x)=\|\kappa^l\|_1\leq \|\kappa^l\|_{\frac{1}{l}}\|1\|_{\frac{1}{1-l}}.
\]    
Clearly,
\[
    \|1\|_{\frac{1}{1-l}}=\left(\int_{\partial K}\mathrm{d}H^{n-1}\right)^{1-l}\leq H^{n-1}(\partial B)^{1-l}.
\]
On the other hand, since $\kappa$ is the density of the absolutely continuous part and the measures are positive we get that
\[
    \|\kappa^l \|_{\frac{1}{l}}=\left(\int_{\partial K}\kappa(K,x)\mathrm{d}H^{n-1}\right)^{l}\leq C_0(K,\mathbb{R}^n)^l= H^{n-1}(\partial B)^l.
\]
\end{proof}

Note that these valuations, $\phi_l$ with $l\in[0,1]$, provide counterexamples for the statement of Lemma \ref{l:counting} for any homogeneity degree $m\leq n-1$ proving that the result is, in this aspect, sharp. We summarize this discussion in the following:

\begin{teorema}
Let $V:\mathcal{K}^n\to\mathbb{R}$ be a translation invariant valuation which is bounded on $\mathcal{K}(B)$. If $V$ is $m$-homogeneous, then $m\in[0,n-1]\cup \{n\}$. Moreover, for each $m\in[0,n-1]\cup\{n\}$, there exists a translation invariant valuation $V$ on $\mathcal{K}^n $ that is bounded on $\mathcal{K}(B)$ and $m$-homogeneous.
\end{teorema}

\section{Some negative results}\label{S:limitations}
In this section, we will show the limitations of the automatic continuity approach in the study of valuations. First we show that, in contrast with our main theorem —the characterization of volume— a similar generalization is not possible for McMullen’s Theorem \ref{teo:McMullenn-1} on $(n-1)$-homogeneous valuations. Afterwards, we will show some negative results regarding attempts to obtain a similar automatic continuity theorem in the context of valuations on convex bodies. We will provide counterexamples that illustrate the failure of certain natural extrapolations of the result to $\mathcal{K}^n$.

First, motivated by our generalization of Hadwiger’s characterization of the volume, we next show that an analogous extension of McMullen’s integral representation for valuations in $\val_{n-1}$ fails when the continuity condition is omitted.
Specifically, the failure of Lemma \ref{l:counting} for $m=n-1$ implies that no analogue of McMullen's Theorem \ref{teo:McMullenn-1} can exist for $(n-1)$-homogeneous, translation invariant and bounded valuations.

\begin{prop}
Let $V$ be a bounded valuation such that there exists a function $f:\mathbb{S}^{n-1}\to\mathbb{R}$ which is integrable with respect to $S_{n-1}(K)$ for every $K\in\mathcal K^n$ and satisfies
\[
    V(K)=\int_{\mathbb{S}^{n-1}}f\mathrm{d}S_{n-1}(K),
\]
for every $K\in\mathcal K^n$. If $V|_{\mathcal{P}^n}=0$, then $V=0$. 
\end{prop}
\begin{proof}
The proof is based on \cite[Theorem 1]{McMullen}, with minor modifications, we include it here for completeness.

Let $P\in \mathcal{P}^n$, we will denote $N(P)$ the set of outer unit normal vectors of the facets of $P$ and let $F(P,u)$ denote the facet that has $u$ as a unit normal vector. Note that for every $P\in \mathcal{P}^n$, $S_{n-1}(P)$ is a linear combination of point masses corresponding to $N(P)$. Therefore,
\[
    V(P)=\int_{\mathbb{S}^{n-1}}f\mathrm{d}S_{n-1}(P)=\sum_{u\in N(P)}H^{n-1}(F(P,u))f(u)=0.
\]

First, we will see that $f$ must be continuous. Let $(v_j)_{j\in\mathbb{N}}\subset\mathbb{S}^{n-1}$ a sequence of unit vectors that converge to $u_0$. Let us assume that $u_0=\frac{1}{\sqrt{n}}(1,\dots,1)$ and consider $u_i=-e_i$ for $1\leq i\leq m$. We define $T$ as the simplex inscribed inside of $\mathbb{S}^{n-1}$ that has $u_0,u_1,\dots,u_n$ as normal vectors. In a similar way, we define $T_j$ the simplex inscribed in $\mathbb{S}^{n-1}$ with $v_j,u_1,\dots,u_n$ as normal vectors. For $j$ big enough (in particular when all the coordinates of $v_j$ are positive) we have that $T_j$ is well defined. It is easy to see that $T_j\xrightarrow[j\to\infty]{}T$ in the Hausdorff metric. Also, it is clear that $F(T_j,u_i)\xrightarrow[j\to\infty]{} F(T,u_i)$ for $1\leq i\leq n$ and that $F(T_j,v_j)\xrightarrow[j\to\infty]{}F(T_j,u_0)$.

Since all $T$ and $T_j$ are polytopes, we get that $\varphi(T)=\varphi(T_j)=0$, and thus
\begin{align*}
        f(v_j)H^{n-1}(F(T_j,v_j))&+\sum_{i=1}^nf(u_i)H^{n-1}(F(T_j,u_i))= \\
        &=f(u_0)H^{n-1}(F(T,u_0))+\sum_{i=1}^n f(u_i)H^{n-1}(F(T,u_i)).
\end{align*}

Now, since $F(T_j,u_i)\to F(T,u_i)$ for $1\leq i\leq n$, $F(T_j,v_0)\to F(T,u_0)$ and $H^{n-1}(F(T,u_0))>0$, we get that $f(v_j)\to f(u_0)$. Thus, we get that $f$ is continuous at $u_0$.

If $u_0$ is not $\frac{1}{\sqrt{n}}(1,\dots,1)$ we can consider a rotation $g$ that moves $u_0$ there. Then $g(v_j)$ would be a sequence that converges to $g(u_0)$ and we could make the same construction and would get the result for any point on the sphere. Thus, $f$ is continuous.

Now, let us show that $f$ is linear: Let $\{e_1,\dots,e_n\}$ be an orthonormal basis of $\mathbb{R}^n$. If $P\in\mathcal{P}^n$ with $P\subset \{x\in\mathbb{R}^n:\langle x,e_i\rangle=0\}$, then
\[
    0=\varphi(P)= f(e_i)H^{n-1}(P)+f(-e_i)H^{n-1}(P)\implies f(e_i)=-f(-e_i).
\]
Let $u=\sum_{i=1}^n\eta_ie_i\in\mathbb{S}^{n-1}$ with $\eta_i\neq0$ for all $i=1,\dots,n$. We take $e_i'=-\mathrm{sign}(\eta_i)e_i$ and a simplex $T$ with $u,e_1',\dots,e_n'$ as its normal vectors. By \cite[Lemma 5.1.1]{Schneider}, taking $F_i=F(T,e_i')$ and $F_0=F(T,u)$, we have
\[
    H^{n-1}(F_0)u+\sum_{i=1}^nH^{n-1}(F_i)e_i'=\vec{0}.
\]
Now, taking the scalar product with $a=\sum_{j=1}^nf(e_j)e_j$ we get
\begin{align*}
    0&=H^{n-1}(F_0)\langle a,u\rangle+\sum_{i=1}^nH^{n-1}(F_i)\langle a,e_i'\rangle\\
    &=H^{n-1}(F_0)\langle a,u\rangle+\sum_{i=1}^nH^{n-1}(F_i)f(e_i'),
\end{align*}
since we already showed that $f(e_i)=-f(-e_i)$. On the other hand,
\[
    0=\varphi(T)=H^{n-1}(F_0)f(u)+\sum_{i=1}^nH^{n-1}(F_i)f(e_i').
\]
Combining both results we get that $f(u)=\langle a,u\rangle$ for any $u\in\mathbb{S}^{n-1}$ with all coordinates different from zero. Finally, from the continuity of $f$ we get the result.
\end{proof}


Let $\phi_0$ be the $(n-1)$-homogeneous bounded translation invariant valuation defined in the previous section which is given by
\[
\phi_0(K)=\int_{\partial K}\chi_{\{x\in \partial K:\kappa (K,x)>0\} }dH^{n-1}.
\]

\begin{coro}
There exists no function $f:\mathbb{S}^{n-1}\to\mathbb{R}$ such that
\[
    \phi_0(K)=\int_{\mathbb{S}^{n-1}}f\mathrm{d}S_{n-1}(K)
\]
for all $K\in\mathcal{K}^n$.
\end{coro}
\begin{proof}
Since $\phi_0|_{\mathcal{P}^n}=0$, by the previous result, it would imply that $\phi_0=0$, but $\phi_0(B)=H^{n-1}(\partial B)$.
\end{proof}

Now, we turn our attention to the search of a result similar to Theorem \ref{teo:CFE} in the context of valuations on $\mathcal{K}^n$. The question is the following one, does the Cauchy Functional equation with respect to Minkowski addition on $\mathcal{K}^n$ exhibit any automatic continuity? In other words, if $V:\mathcal{K}^n\to\mathbb{R}$ is a bounded valuation satisfying
\[
    V(K+L)=V(K)+V(L)
\]
must $V$ be continuous?

The answer is negative, as the following shows. Consider the mapping $\phi(K)=S_1^s(K)(\mathbb{S}^{n-1})$, where $S_1^s(K)$ is the singular part of $S_1(K)$ with respect to the Haar measure of $\mathbb{S}^{n-1}$. One can interpret this map as the composition of three mappings, $S_1$, $\mu\mapsto\mu^s$ and $\langle \cdot, 1\rangle$ in the following way
\[
    \begin{array}{rccccccc}
        \phi: & \mathcal{K}^n&\xrightarrow[]{S_1} &\mathcal{M}(\mathbb{S}^{n-1}) & \xrightarrow[]{^s} &\mathcal{M}(\mathbb{S}^{n-1}) &\xrightarrow[]{\langle\cdot, 1\rangle}&\mathbb{R} \\
         & K&\longmapsto &S_1(K) & \longmapsto & S_1^s(K) & \longmapsto & S_1^s(K)(\mathbb{S}^{n-1}).
    \end{array}
\]
The mapping $S_1:\mathcal{K}^n\to\mathcal{M}(\mathbb{S}^{n-1})$ is well-known to be Minkowski-additive, that is
\[
    S_1(K+L)=S_1(K)+S_1(L).
\]
Similarly, the mapping $\mu\mapsto\mu^s$, which gives the singular part of any measure with respect to the Haar measure, is linear. This follows from the uniqueness of the Lebesgue decomposition \cite[Theorem 4.3.2]{Cohn}. Finally, the last mapping is defined as
\[
    \langle \cdot , 1\rangle(\mu)=\int_{\mathbb{S}^{n-1}}1\mathrm{d}\mu=\mu(\mathbb{S}^{n-1}).
\]
Clearly, this mapping is also linear, so $\phi$ is a Minkowski-additive valuation. Note however, this mapping is not continuous since $\phi(B)=0$ while $\phi(P)=S_1(P)(\mathbb{S}^{n-1})$ for any $P\in\mathcal{P}^n$. Also, this valuation is bounded: for $K\subset B$ we have that
\[
    S_1^s(K)(\mathbb{S}^{n-1})\leq S_1(K)(\mathbb{S}^{n-1})\leq S_1(B)(\mathbb{S}^{n-1}).
\]
Thus, an automatic continuity result as that of Theorem \ref{teo:CFE} is not possible in this setting.

Besides Minkowski addition, Blaschke addition is another semigroup operation inside $\mathcal{K}^n$, that has proven to be quite useful. For the definition we refer the reader to \cite{Schneider} and \cite{McMullen}. We will see that a valuation that is Blaschke additive and bounded does not have to be continuous. To exhibit such an example, we will make use of McMullen's Characterization of $(n-1)$-homogeneous continuous and translation invariant valuations, Theorem \ref{teo:McMullenn-1}.  The idea is to build valuations via the surface area measure, $S_{n-1}$, which is Blaschke additive, and then modify the function in the integral representation to obtain a valuation that is Borel measurable and bounded but not continuous.

\begin{prop}\label{prop:borel}
Let $f$ be a bounded Borel function on the sphere $\mathbb{S}^{n-1}$. Then the following valuation
\[
    \phi_f(K)=\int_{\mathbb{S}^{n-1}}f(x)\mathrm{d}S_{n-1}(K,x)
\]
is in $\val_{n-1}$ if and only if $f$ is continuous.
\end{prop}

\begin{proof}
Assume that $\phi_f\in\val_{n-1}$. By McMullen's Theorem \ref{teo:McMullenn-1}, we get that there exists a continuous function $g$ on the sphere such that
\[
    \phi_f(K)=\int g\mathrm{d}S_{n-1}(K)=\int f\mathrm{d}S_{n-1}(K).
\]
for any $K\in\mathcal{K}^n$. Then we get that 
\[
    \int (f-g)\mathrm{d}S_{n-1}(K)=0
\]
for any $K\in\mathcal{K}^n$. Since any balanced measure can be written as the difference of two surface area measures (check \cite[Theorem 3]{McMullen})
we get that
\[
    \int (f-g)\mathrm{d}\mu=0
\]
for any balanced measure $\mu$.

Now, let $L=\{\langle u, \cdot \rangle, u\in\mathbb{R}^n\}\subset C(\mathbb{S}^{n-1})$ be the subspace of linear functions. Note that this has finite dimension, which means that it is closed in all vector topologies. Now, the polar of $L$ is exactly the set of balanced measures and $f-g$ is in the polar of balanced measures. By the Bipolar Theorem (see \cite[Theorem 3.38]{Fabian}) we get that $f-g$ is in $L$, that is, $f-g$ is continuous and, since $g$ was continuous, we get that $f$ is continuous.

The converse result is immediate from the fact (see \cite[Theorem 4.2.1]{Schneider}) that convergence of convex bodies $K_j\xrightarrow[]{}K$ in the Hausdorff metric implies that, for every continuous $f:\mathbb{S}^{n-1}\to\mathbb{R}$, 
\[
    \int_{\mathbb{S}^{n-1}}f\mathrm{d}S_{n-1}(K_j)\xrightarrow{}\int_{\mathbb{S}^{n-1}}f\mathrm{d}S_{n-1}(K).
\]
\end{proof}

\begin{lema}
Let $f$ be a bounded Borel function on the sphere, then the following valuation is measurable
\[
    \phi_f(K)=\int_{\mathbb{S}^{n-1}}f\mathrm{d}S_{n-1}(K).
\]
\end{lema}
\begin{proof}
First we will see that for any Borel subset $A\subset \mathbb{S}^{n-1}$, its characteristic function, $\chi_A$, satisfies the lemma. Let $\mathcal{D}=\{A\in\mathcal{B}(\mathbb{S}^{n-1}): \phi_{\chi_A}\text{ is measurable}\}$, we will see that the Borel sets $\mathcal{B}(\mathbb{S}^{n-1})$ is equal to $\mathcal{D}$ by checking that $\mathcal{D}$ is a Dynkin Class that contains the open sets.

It is clear, that both $\emptyset$ and $\mathbb{S}^{n-1}\in\mathcal{D}$. We now check that the open sets are in $\mathcal{D}$. Consider a proper non-empty open set $A\subset\mathbb{S}^{n-1}$. It is well known that its indicator function $\chi_A$ is the pointwise limit of an increasing sequence of continuous functions $f_m$ with $0\leq f_m\leq f_{m+1}\leq \chi_A$ on the sphere. We define the following mapping 
\[
    \phi_{\chi_A}(K)=\int \chi_A(x)\mathrm{d}S_{n-1}(K,x),
\]
which is obviously a translation invariant valuation and $(n-1)$-homogeneous. By McMullen's Theorem, we know that
\[
    \phi_m(K)\coloneqq \int f_m(x)\mathrm{d}S_{n-1}(K,x)
\]
are continuous valuations and, in particular, Borel. Then, by the Dominated Convergence Theorem, given $K\in\mathcal{K}^n$ we get
\[
    \lim_{m\to\infty}\phi_m(K)=\int \lim_{m\to\infty}f_m(x)\mathrm{d}S_{n-1}(K,x)=\phi_{\chi_A}(K).
\]
Thus, $\phi_{\chi_A}$ is the pointwise limit of Borel functions and is itself Borel, showing that the open sets are inside $\mathcal{D}$.

Now, given $A_1,A_2\in\mathcal{D}$ with $A_1\subset A_2$, it is easy to see that
\[
    \phi_{\chi_{A_2}}-\phi_{\chi_{A_1}}=\phi_{\chi_{A_2\setminus A_1}},
\]
which implies that $A_2\setminus A_1\in\mathcal{D}$. Next, let $\{A_j\}_{j\in\mathbb{N}}\subset\mathcal{D}$ and $A_j\subset A_{j+1}$ for every $j\in\mathbb{N}$. It is clear that
\[
    \chi_{\bigcup_{j\in\mathbb{N}}A_j}(x)=\lim_{m\to\infty}\chi_{A_m}(x),
\]
where $\chi_{A_m}$ are measurable, non negative and increasing. Thus, by the Monotone Convergence Theorem, for any $K\in\mathcal{K}^n$ 
\[
    \phi_{\chi_{\bigcup_{j\in\mathbb{N}}A_j}}(K)=\lim_{m\to\infty} \phi_{\chi_{A_m}}(K).
\]
Therefore, $\phi_{\chi_{\bigcup_{j\in\mathbb{N}}A_j}}$ is the pointwise limit of measurable functions and thus measurable. These facts show that $\mathcal{D}$ is a Dynkin class that contains all open sets. By a classical result of Dynkin classes (check \cite[Chapter 1.6]{Cohn}), all Borel sets belong to $\mathcal{D}$.

Now that we have established the result for characteristic functions, it clearly extends to simple functions, and by standard measure theoretic arguments the general case follows.
\end{proof}

Note that combining the two results above shows that if $f:\mathbb{S}^{n-1}\to\mathbb{R}$ a Borel and bounded function that is not continuous, then
\[
    \phi_f(K)=\int_{\mathbb{S}^{n-1}}f\mathrm{d}S_{n-1}(K)
\]
is a translation-invariant, $(n-1)$-homogeneous, Blaschke additive valuation that is bounded and Borel, but it is not continuous.

\printbibliography[heading=bibintoc]

\appendix

\section{General valuation on parallelotopes }\label{S:AppendixA}

In Section \ref{S:par}, we proved an automatic continuity result (Theorem \ref{teo:valuaciones_continuidad}) for translation invariant valuations on $\Pa(\{e_i\}_{i=1}^n)$. In particular, we established the characterization Theorem \ref{t:McMullen-pa} that led to the proof of the aforementioned result. In this appendix, we present a different characterization theorem \ref{t:val on pa} for all valuation on $\Pa(\{e_i\}_{i=1}^n)$. From this result, we can provide several examples that show automatic continuity may fail on $\Pa(\{e_i\}_{i=1}^n)$. More specifically, we provide examples of bounded Borel valuations which are not continuous.

It is not surprising that unlike in the situation of translation invariant valuations on $\Pa(\{e_i\}_{i=1}^n)$, where the valuation only depends on the side lengths of a paralellotope, the characterization for arbitrary valuations on $\Pa(\{e_i\}_{i=1}^n)$ depends on the actual positions of its extreme points. We will prove first the case $n=1$, which contains the core of the idea. Note that in this case, $\Pa(1)=\mathcal{K}^1$.

\begin{lema}\label{lema:para1d}
If a mapping $V:\mathcal{K}^1\to \mathbb{R}$ is a valuation then there exist two functions $f,g:\mathbb{R}\to\mathbb{R}$ such that for any $a\leq b$
\[
    V([a,b])=f(a)+g(b).
\]
Conversely, for any two functions $f,g:\mathbb{R}\to\mathbb{R}$, the mapping $V:\mathcal{K}^1\to\mathbb{R}$ defined as
\[
    V([a,b])\coloneqq f(a)+g(b)
\]
for $a\leq b$, is a valuation.
\end{lema}

\begin{proof}
Let $V$ be a valuation on $\mathcal{K}^1$. Let $V(a)$ denote $V([a,a])=V(\{a\})$. Clearly, for $a<b$ we get
\[
    V[-m,b]+V[a,m]=V[-m,m]+V[a,b]=V[-m,0]+V[0,m]-V(0)+V[a,b]
\]
for any $m>|a|,|b|$. Thus,
\[
    V[a,b]=V[a,m]-V[0,m] + V[-m,b]-V[-m,0]+V(0).
\]
We define the functions $f_m:(-m,m)\to\mathbb{R}$ and $g_m:(-m,m)\to\mathbb{R}$ as
\[
    f_m(a)=V[a,m]-V[0,m] \quad\text{ and }\quad g_m(b)=V[-m,b]-V[-m,0].
\]
Therefore,
\[
    V[a,b]=f_m(a)+g_m(b)+V(0).
\]
Note that for fixed $a,b\in\mathbb R$, the values $f_m(a)$ and $g_m(b)$ are independent of $m$ as long as $m$ is large enough, hence we can define
\[
    f(a)\coloneqq\lim_{m\to\infty} f_m(a),\quad g(b)\coloneqq\lim_{m\to\infty}g_m(b),
\]
and 
\[
    V[a,b]=f(a)+g(b)+V({0}).
\]
For $a=b$, the same computations show that
\[
    V(a)=f(a)+g(a)+V(0).
\]
For completeness, we include working definitions of the previous functions:
\[
    f(a)=\left\{\begin{array}{ll}
        V[a,0]-V(0) & \text{if }a<0 \\
        0 & \text{if }a=0\\
        -V[0,a]+V(a)&\text{if }a>0, 
    \end{array}  \right. \quad g(b)=\left\{\begin{array}{ll}
        -V[b,0]+V(b) & \text{if }b<0 \\
        0 & \text{if }b=0\\
        V[0,b]-V(0)&\text{if }b>0.
    \end{array}  \right.
\]
Finally, the constant $V(0)$ can be included in any of the two previous functions and we get the result.

Conversely, if $f$ and $g$ are any functions from $\mathbb{R}$ to $\mathbb{R}$, define the function $V:\mathcal{K}^1\to\mathbb{R}$ as 
\[
    V[a,b]\coloneqq f(a)+g(b) \quad \text{and}\quad V(\{a\})=f(a)+g(a).
\]
Take $K=[a^K,b^K]$ and $L=[a^L,b^L]$ such that $K\cup L\in\Pa(\{e_1\})$. Note that the union is convex only when $K\cap L\neq\emptyset$. In that case,
\[
    K\cup L=[\min\{a^K,a^L\},\max\{b^K,b^L\}],\quad K\cap L=[\max\{a^K,a^L\},\min\{b^K,b^L\}].
\]
Thus,
\begin{align*}
    V(K\cup L)+V(K\cap L)&= f(\min\{a^K,a^L\})+g(\max\{b^K,b^L\})+ \\
    &\quad  + f(\max\{a^K,a^L\})+g(\min\{b^K,b^L\})= \\
    & =f(a^K)+f(a^L)+g(b^K)+g(b^L)=V(K)+V(L).
\end{align*}
\end{proof}

\begin{rem}
Note that $f$ and $g$ in previous lemma are not unique. In the proof presented, given a valuation we choose a precise definition of $f$ and $g$ but, clearly, we could replaced them by $f-c$ and $g+c$ for any $c\in\mathbb{R}$.
In fact, let $f,f',g,g'$ be functions from $\mathbb{R}$ to itself such that 
    \[
        f(a)+g(b)=f'(a)+g'(b)
    \]
for any $a\leq b$. Then,
    \[
        f(a)-f'(a)=g'(b)-g(b)
    \]
which implies that $f-f'=g'-g=c$ with $c\in\mathbb{R}$. This shows that, in the one dimensional case, the only possible ambiguity in the choice of $f$ and $g$ is the addition or subtraction of a constant; however, this will not be true in higher dimensions. This essential uniqueness in the case $n=1$ implies that properties of the valuation will directly  translate to properties of any representation $f$ and $g$. This fact can then be used as in the following corollary.
\end{rem}

\begin{coro}
A valuation $V:\mathcal{K}^1\to\mathbb{R}$ is continuous (respectively, Borel) if and only if both $f$ and $g$ are continuous (resp. Borel).
\end{coro}

It follows immediately (choosing for instance $f$ and $g$ Borel not continuous functions) that an analog of the Automatic Continuity Theorem \ref{teo:CFE} does  not hold in this setting.  It is also interesting to note that, in the one-dimensional situation, every Borel valuation is an iterated pointwise limit of continuous valuations, in the same way as every Borel function on a metric space can be reached by iterated pointwise limits of continuous functions (cf. \cite[Theorem 11.6]{kechris}). 


In the proof of Lemma \ref{lema:para1d} we embedded the interval $[a, b]$ into a larger symmetric interval $[-m, m]$. We then subdivided this larger interval using the endpoints of $[a, b]$ as reference points. This allowed us to decompose the valuation into two parts, each depending only on one of the endpoints. The same idea works for higher dimensional parallelotopes. In these case the set of extreme points will be given by the vertices of the parallelotope. Although this method is geometrically elegant, it is difficult to implement as it can become overly complex and confusing when trying to keep track of all different intersections. To address this, we will use a different approach based on the use of Lemma \ref{lema:unioncuad} and reformulate the problem by translating it from $\Pa(\{e_i\}_{i=1}^n)$ to $\left(\mathcal{K}^{1}\right)^n=\left(\Pa(1)\right)^{n}$. This reformulation of the problem will help us simplify the proof with a straightforward induction on the dimension. 

\begin{lema} \label{lema:reform.valuation}
Let $V:\Pa(\{e_i\}_{i=1}^n)\to\mathbb{R}$ be a mapping and let $F:\left(\mathcal{K}^{1}\right)^n\to\mathbb{R}$ be such that
\[
    V\left(\sum_{i=1}^n[a_i,b_i]e_i\right)=F([a_1,b_1],\dots,[a_{n},b_n]).
\]
Then, $V$ is a valuation in $\Pa(\{e_i\}_{i=1}^n)$ if and only if for any $j\in\{1,\dots,n\}$ given $a_i\leq b_i$ real numbers for all $i\neq j$ the function $f_j:\mathcal{K}^{1}\to\mathbb{R}$ defined as
\[
    f_j([x,y])\coloneqq F([a_1,b_1],\dots,[a_{j-1},b_{j-1}],[x,y],[a_{j+1},b_{j+1}],\dots,[a_n,b_n])
\]
is a valuation.
\end{lema}
\begin{proof}
Let $j\in\{1,\dots,n\}$ and fix $a_i\leq b_i$ real numbers for all $i\neq j$. Then, for any $a\leq b$, $f_j$ satisfies
\[
    V\left(\sum_{\substack{
         i=1  \\
          i\neq j
    }}^n [a_i,b_i]e_i+[a,b]e_j\right)=f_j([a,b]).
\]
Now, notice that for $x,y,x'$ and $y'$ real numbers satisfying $[x,y]\cap[x',y']\neq\emptyset$, we can define
\[
    K=\sum_{\substack{
         i=1  \\
          i\neq j
    }}^n [a_i,b_i]e_i+[x,y]e_j\quad \text{and}\quad L=\sum_{\substack{
         i=1  \\
          i\neq j
    }}^n [a_i,b_i]e_i+[x',y']e_j
\]
Therefore,
\begin{align*}
    &K\cup L=\sum_{\substack{
         i=1  \\
          i\neq j
    }}^n [a_i,b_i]e_i+\left[\min\{x,x'\},\max\{y,y'\}\right]e_j,\\
    &K\cap L=\sum_{\substack{
         i=1  \\
          i\neq j
    }}^n [a_i,b_i]e_i+\left[\max\{x,x'\},\min\{y,y'\}\right]e_j.
\end{align*}
This implies 
\[
    \begin{array}{ll}
        f_j([x,y])=V(K),  & f_j([\min\{x,x'\},\max\{y,y'\}])=V(K\cup L), \\
        f_j([x',y'])=V(L), & f_j([\max\{x,x'\},\min\{y,y'\}])=V(K\cap L).
    \end{array}
\]
Thus, if $V$ is a valuation, clearly we have that $f_j$ satisfies the property.

Conversely, if all $f_j$ satisfy the property, by Lemma \ref{lema:unioncuad}, we only have two cases to consider. If $K\subset L$ or $L\subset K$ the equality is trivially obtained, but the second case is precisely the equality given by $f_j$.
\end{proof}
From now on, to simplify notation, we will omit the interval notation in the previous mapping $F$, that is, we will write
\[
    F([a_1,b_1],\dots,[a_{n},b_n])=F(a_1,b_1,\dots,a_n,b_n).
\]
\begin{teorema}\label{t:val on pa}
If a mapping $V:\Pa(\{e_i\}_{i=1}^n)\to\mathbb{R}$ is a valuation then there exist functions $F_{i_1,\dots,i_n}:\mathbb{R}^n\to\mathbb{R}$ with $i_j\in\{1,2\}$ for $j\in\{1,\dots,n\}$, such that
\[
    V\left(\sum_{i=1}^n[a_i^{(1)},a_i^{(2)}]e_i\right)=\sum_{i_1,\dots,i_n=1}^2F_{i_1,\dots,i_n}(a_1^{(i_1)},\dots,a_n^{(i_n)})
\]
for $a_i^{(1)}\leq a_i^{(2)}$ in $\mathbb{R}$. Conversely, for any collection of functions $F_{i_1,\dots,i_n}:\mathbb{R}^n\to\mathbb{R}$ with $i_j\in\{1,2\}$ for $j\in\{1,\dots,n\}$, the mapping $V:\Pa(\{e_i\}_{i=1}^n)\to\mathbb{R}$ defined as 
\[
     V\left(\sum_{i=1}^n[a_i^{(1)},a_i^{(2)}]e_i\right)=\sum_{i_1,\dots,i_n=1}^2F_{i_1,\dots,i_n}(a_1^{(i_1)},\dots,a_n^{(i_n)})
\]
for $a_i^{(1)}\leq a_i^{(2)}$ in $\mathbb{R}$, is a valuation. 
\end{teorema}

\begin{proof}
We will prove this by induction on the dimension $n$. Lemma \ref{lema:para1d} corresponds to the case $n=1$. Let us now assume the statement holds for $n-1$ and we will prove it for $n$. 

Let $V:\Pa(\{e_i\}_{i=1}^n)\to\mathbb{R}$ be a valuation. For fixed $a_i^{(1)}\leq a_i^{(2)}$ with $1\leq i\leq n-1$, let $P_{n-1}=P_{n-1}(a_1^{(1)},a_1^{(2)},\ldots, a_{n-1}^{(1)},a_{n-1}^{(2)})\coloneqq\sum_{i=1}^{n-1}[a_i^{(1)},a_i^{(2)}]e_i$ and define the following mapping $V':\mathcal{K}^1\to\mathbb{R}$ as
\[
    V'([a_{n}^{(1)},a_n^{(2)}])=V\left(P_{n-1} + [a_n^{(1)},a_n^{(2)}]e_n  \right), 
\]
for $a_n^{(1)}\leq a_n^{(2)}$. It is easy to see that this mapping is a valuation, so by Lemma \ref{lema:para1d}, there are functions $F,G:\mathbb{R}^{2n-1}\to\mathbb{R}$ and $H:\mathbb{R}^{2n-2}\to\mathbb{R}$ such that
\begin{align*}
    V\left(\sum_ {i=1}^n[a_i^{(1)},a_i^{(2)}]e_i\right)&=F(a_1^{(1)},\dots,a_{n-1}^{(2)},a_n^{(1)})+G(a_1^{(1)},\dots,a_{n-1}^{(2)},a_n^{(2)})+\\
    &\quad+H(a_1^{(1)},\dots,a_{n-1}^{(2)}).
\end{align*}

As in the proof of Lemma \ref{lema:para1d}, these functions can be computed as 
    \begin{align*}
         F(a_1^{(1)},\dots,a_{n-1}^{(2)},a_n^{(1)})&=\!\lim_{m\to\infty}V\left(P_{n-1}\! + [a_n^{(1)},m]e_n\right)\!-V\left( P_{n-1} + [0,m]e_n\right)  \\
         G(a_1^{(1)},\dots,a_{n-1}^{(2)},a_n^{(2)})&=\!\lim_{m\to\infty}V\left(P_{n-1}\!+[-m,a_n^{(2)}]e_n\right)\!-V\left( P_{n-1} + [-m,0]e_n\right)\\
          H(a_1^{(1)},\dots,a_{n-1}^{(2)})&=V\left( P_{n-1}\right).
    \end{align*}
    
Now, note that for fixed $a_n^{(1)}$, for any $1\leq j\leq n-1$, $b_j^{(1)}\leq b_j^{(2)}$ and $c_j^{(1)}\leq c_j^{(2)}$ with $c_j^{(1)}\leq b_j^{(2)}$ and $b_j^{(1)}\leq c_j^{(2)}$, it is straightforward to check (using the definition of $F$ and the fact that $V$ is a valuation) that
\begin{align*}
    F(a_1^{(1)},\dots, b_j^{(1)},b_j^{(2)},\dots,a_{n-1}^{(2)},a_n^{(1)})+F(a_1^{(1)},\dots, c_j^{(1)},c_j^{(2)},\dots,a_{n-1}^{(2)},a_n^{(1)})=\\
    F(a_1^{(1)},\dots, c_j^{(1)},b_j^{(2)},\dots,a_{n-1}^{(2)},a_n^{(1)})+F(a_1^{(1)},\dots, b_j^{(1)},c_j^{(2)},\dots,a_{n-1}^{(2)},a_n^{(1)}).
\end{align*}
Similarly, for fixed $a_n^{(2)}$ we get that $G$ also satisfies
\begin{align*}
    G(a_1^{(1)},\dots, b_j^{(1)},b_j^{(2)},\dots,a_{n-1}^{(2)},a_n^{(2)})+G(a_1^{(1)},\dots, c_j^{(1)},c_j^{(2)},\dots,a_{n-1}^{(2)},a_n^{(2)})=\\
     G(a_1^{(1)},\dots, c_j^{(1)},b_j^{(2)},\dots,a_{n-1}^{(2)},a_n^{(2)})+G(a_1^{(1)},\dots, b_j^{(1)},c_j^{(2)},\dots,a_{n-1}^{(2)},a_n^{(2)}).
\end{align*}
In other words, for fixed $a_n^{(1)}$ and $a_n^{(2)}$ and using Lemma \ref{lema:reform.valuation}, the functions $F$ and $G$ on the first $n-1$ variables define valuations on $\Pa(\{e_i\}_{i=1}^{n-1})$, that is, viewing $\mathbb{R}^{n-1}$ as the hyperplane generated by $\{e_1,\dots,e_{n-1}\}$
It is also clear that $H$ is also a valuation in $\Pa(\{e_i\}_{i=1}^{n-1})$. Thus, by the induction hypothesis, the conclusion follows. The converse follows clearly as the $n=1$ case.
\end{proof}

\begin{rem}
As in the case $n=1$, this characterization easily yields counterexamples to an automatic continuity result as Theorem \ref{teo:CFE} for valuations on $\Pa(\{e_i\}_{i=1}^n)$ for any $n$. For instance, take all functions to be identically $0$ except $F_{1,\dots,1}$, and choose this function to be Borel and bounded but not continuous. This gives a Borel, bounded and not continuous valuation $V$. 
\end{rem}

\begin{rem}
As in the case $n=1$, this representation is not unique. In fact, when $n>1$ we have more freedom in choosing the representation and the continuity of $V$ no longer implies that all the corresponding representation functions are continuous. An example of this can be seen for $n=2$: Take $F_{1,2}(a_1^{(1)},a_2^{(2)})=f(a_2^{(2)})$, $F_{2,2}(a_2^{(1)},a_2^{(2)})=-f(a_2^{(2)})$ and $F_{1,1}=F_{2,1}=0$. Then, $V=0$ but we can choose $f$ to make the functions $F_{1,2}$ and $F_{2,2}$ not continuous. However, in the proof of the Lemma, if $V$ is continuous the functions we have defined are continuous, so at least one continuous representation exists. 
\end{rem}

\end{document}